\documentclass[12pt]{scrartcl}
\usepackage{a4}
\usepackage{amsthm}
\usepackage{amsmath}
\usepackage{amssymb}
\usepackage{amsfonts}
\usepackage{mathrsfs}
\usepackage{dsfont}
\usepackage{latexsym}
\usepackage{color}
\usepackage{bbm,exscale}
\definecolor{Myblue}{rgb}{0,0,0.6}
\usepackage[a4paper,colorlinks,citecolor=Myblue,linkcolor=Myblue,urlcolor=Myblue,pdfpagemode=None]{hyperref}
\usepackage[square,numbers,sort&compress]{natbib}
\usepackage[all,cmtip]{xy}

  \vfuzz \hfuzz
\makeatletter
\newcommand{\raisemath}[1]{\mathpalette{\raisem@th{#1}}}
\newcommand{\raisem@th}[3]{\raisebox{#1}{$#2#3$}}
\makeatother

\newcommand{\E}{\text{e}}
\newcommand{\I}{\text{i}}

\newcommand{\B}{\mathcal{B}}
\newcommand{\Pc}{\mathcal{P}}

\newcommand{\C}{\mathds{C}}

\newcommand{\PP}{\mathds{P}}
\newcommand{\Q}{\mathds{Q}}

\newcommand{\Z}{\mathds{Z}}
\newcommand{\DD}{\mathds{D}}

\def\1{\ifmmode\mathrm{1\!l}\else\mbox{\(\mathrm{1\!l}\)}\fi}

\newcommand{\be}{\begin{equation}}
\newcommand{\ee}{\end{equation}}
\newcommand{\bes}{\begin{equation*}}
\newcommand{\ees}{\end{equation*}}

\newcommand{\Hom}{\operatorname{Hom}}
\newcommand{\End}{\operatorname{End}}
\newcommand{\modu}{\operatorname{mod}}

\newcommand{\hmf}{\operatorname{hmf}}
\newcommand{\hmfgr}{\operatorname{hmf}^{\textrm{gr}}}

\newcommand{\ev}{\operatorname{ev}}
\newcommand{\tev}{\widetilde{\operatorname{ev}}}
\newcommand{\coev}{\operatorname{coev}}
\newcommand{\tcoev}{\widetilde{\operatorname{coev}}}
\def\lra{\longrightarrow}

\DeclareMathOperator{\str}{str}

\DeclareMathOperator{\Res}{Res}
\newcommand{\diml}{\dim_{\mathrm{l}}}
\newcommand{\dimr}{\dim_{\mathrm{r}}}
\newcommand{\Dl}{\mathcal D_{\mathrm{l}}}
\newcommand{\Dr}{\mathcal D_{\mathrm{r}}}

\newcommand\arxiv[2]      {\href{http://arXiv.org/abs/#1}{#2}}
\newcommand\doi[2]        {\href{http://dx.doi.org/#1}{#2}}
\newcommand\httpurl[2]    {\href{http://#1}{#2}}

\allowdisplaybreaks

\deffootnote[1em]{1em}{1em}
{\textsuperscript{\thefootnotemark}}

\theoremstyle{definition}
\newtheorem{definition}{Definition}
\newtheorem{defthm}[definition]{Definition and Theorem}
\newtheorem{proposition}[definition]{Proposition}
\newtheorem{theorem}[definition]{Theorem}
\newtheorem{lemma}[definition]{Lemma}
\newtheorem{corollary}[definition]{Corollary}
\newtheorem{remark}[definition]{Remark}

\newtheorem{conjecture}[definition]{Conjecture}

\numberwithin{equation}{section}
\numberwithin{definition}{section}
\numberwithin{figure}{section}

\newcommand\void[1]{}

\begin{document}

\title{Orbifold equivalent potentials}
\author{Nils Carqueville$^*$ \quad Ana Ros Camacho$^\dagger$ \quad Ingo Runkel$^\dagger$
\\[0.5cm]
  \normalsize{\tt \href{mailto:nils.carqueville@scgp.stonybrook.edu}{nils.carqueville@scgp.stonybrook.edu}} \\
 \normalsize{\tt \href{mailto:ana.ros.camacho@uni-hamburg.de}{ana.ros.camacho@uni-hamburg.de}}
 \\
  \normalsize{\tt \href{mailto:ingo.runkel@uni-hamburg.de}{ingo.runkel@uni-hamburg.de}}\\[0.1cm]
  {\normalsize\slshape $^*$Simons Center for Geometry and Physics,}\\[-0.1cm]
  {\normalsize\slshape Stony Brook, NY 11794, USA}\\[-0.1cm]
  {\normalsize\slshape $^\dagger$Fachbereich Mathematik, Universit\"{a}t Hamburg,}\\[-0.1cm]
  {\normalsize\slshape Bundesstra\ss e 55, 20146 Hamburg, Germany}\\[-0.1cm]
}
\date{}
\maketitle

\vspace{-12.8cm}
\hfill {\scriptsize Hamburger Beitr\"age zur Mathematik 493}

\vspace{-1.0cm}

\hfill {\scriptsize ZMP-HH/13-20}

\vspace{12cm}

\begin{abstract}
To a graded finite-rank matrix factorisation of the difference of two homogeneous potentials one can assign two numbers, the left and right quantum dimension. 
The existence of such a matrix factorisation with non-zero quantum dimensions defines an equivalence relation between potentials, giving rise to non-obvious equivalences of categories.

Restricted to ADE singularities, the resulting equivalence classes of potentials are those of type $\{ \mathrm{A}_{d-1} \}$ for~$d$ odd, $\{ \mathrm{A}_{d-1} , \mathrm{D}_{d/2+1} \}$ for~$d$ even but not in $\{12,18,30\}$, and $\{ \mathrm{A}_{11}, \mathrm{D}_{7}, \mathrm{E}_{6} \}$, $\{ \mathrm{A}_{17}, \mathrm{D}_{10}, \mathrm{E}_{7} \}$ and $\{  \mathrm{A}_{29}, \mathrm{D}_{16}, \mathrm{E}_{8} \}$. 
This is the result expected from two-dimensional rational conformal field theory, and it directly leads to new descriptions of and relations between the associated (derived) categories of matrix factorisations and Dynkin quiver representations. 
\end{abstract}

\thispagestyle{empty}
\newpage

\tableofcontents

\section{Introduction and summary}\label{sec:introduction}

Let $k \subset \C$ be a field. 
By a \textsl{homogeneous potential in~$m$ variables} we mean a polynomial $V \in k[x_1,\dots,x_m]$ such that 
\be 
V \big( \lambda^{|x_1|} x_1, \ldots, \lambda^{|x_m|} x_m \big) 
= 
\lambda^{2} V( x_1, \ldots, x_m )
\quad \text{for all } \lambda \in \C^\times
\ee
and 
\be
\dim_k \! \big( k[x_1,\dots,x_m] / (\partial_{x_1} V, \ldots, \partial_{x_m} V) \big) < \infty 
\ee
where $k[x_1,\dots,x_m]$ is viewed as a graded ring by assigning degrees $|x_i| \in \Q_+$ to the variables~$x_i$. 
As we will only be concerned with homogeneous potentials, we shall refer to them simply as \textsl{potentials}. 

The purpose of this paper is two-fold. Firstly, we will define an equivalence relation on the set $\Pc_k$ of all potentials, where the number of variables in the polynomial ring is allowed to vary, but where the field $k$ is kept fixed. Secondly, in the case $k=\C$ we will list all equivalence classes in the restricted set of simple singularities, leading to new equivalences of categories. 

In the remainder of this section we describe these two points in more detail, leaving technical details and supplementary discussions for Sections~\ref{sec:proofs} and~\ref{sec:CFTcomparison}.

\subsection{Orbifold equivalence of potentials}\label{subsec:introductionpart1}

The equivalence relation is defined by the existence of a matrix factorisation with certain properties. 

We will write~$x$ for the sequence of variables $x_1,\dots,x_m$, and we pick a potential~$V \in k[x]$. 
A \textsl{matrix factorisation of~$V$} is a $\Z_2$-graded free $k[x]$-module~$M$ together with an odd $k[x]$-linear endomorphism $d_M$, the \textsl{twisted differential}, such that 
$
	d_M \circ d_M = V \cdot 1_M 
$. 
Let us denote the two homogeneous subspaces of $M$ with respect to the $\Z_2$-grading by $M^0$ and $M^1$, so that $M = M^0 \oplus M^1$. We call $M \equiv (M,d_M)$ a \textsl{graded matrix factorisation} if $M^0$ and $M^1$ are in addition $\Q$-graded,\footnote{%
Often one rescales the degrees $|x_i|$ by their least common denominator, so that one can work with a $\Z$-grading. For our purposes the present convention, where the degrees are normalised such that each potential has degree 2, is more convenient (the grading is then better behaved when taking differences of potentials). Note that a rescaled $\Z$-graded matrix factorisation gives a $\Q$-graded one, but not necessarily vice versa.} 
and if the following properties are satisfied:
(i) Acting with~$x_i$ is an endomorphism of degree $|x_i|$ with respect to the $\Q$-grading on $M$. Consequently, $V$ has degree 2.
(ii) $d_M$ has degree 1 with respect to the $\Q$-grading on $M$.

There are categories of (graded) matrix factorisations whose morphisms are (homogeneous) even linear maps up to homotopy with respect to the twisted differential. These are idempotent complete as they are triangulated and have arbitrary coproducts \cite{NeemanBook}. We denote by $\hmf(k[x],V)$ the idempotent closure of the full subcategory of finite-rank matrix factorisations; this means that its objects are homotopy equivalent to direct summands of finite-rank matrix factorisations. Furthermore, $\hmfgr(k[x],V)$ is the full subcategory of graded matrix factorisations which are homotopy equivalent to finite-rank ones; $\hmfgr(k[x],V)$ is also idempotent complete, see \cite[Lem.\,2.11]{kst0511155}. 

Let now $V(x)$, $W(y) \in \Pc_k$ be two potentials, where $x=x_1,\dots,x_m$ and $y = y_1,\dots,y_n$. Then $W(y)-V(x)$ is a potential in the ring $k[x,y]$. Let $X = (X,d_X)$ be a finite-rank graded matrix factorisation of $W(y)-V(x)$. The even and odd parts of $X = X^0 \oplus X^1$ necessarily have the same rank, say~$r$, and we may as well assume that $X^0 = X^1 = k[x,y]^r$. The twisted differential is then given by two $r$-by-$r$ matrices $d^0$, $d^1$ with entries in $k[x,y]$. We think of these as $k[x,y]$-linear maps $d^0 : X^0 \to X^1$ and $d^1 : X^1 \to X^0$, and we call $(X,d_X)$ \textsl{rank-}$r$. 

We assign two numbers to~$X$, the \textsl{left and right quantum dimension}~\cite{cr1006.5609, cm1208.1481} 
\begin{align}\label{eq:MF-q-dim}
\dim_{\textrm{l}}(X) 
& = 
(-1)^{\binom{m+1}{2}}\Res \left[ \frac{ \str\big( \partial_{x_1} d_{X}\ldots \partial_{x_m} d_{X} \,  \partial_{y_1} d_{X}\ldots \partial_{y_n} d_{X}\big) \operatorname{d}\! y}{\partial_{y_1} W, \ldots, \partial_{y_n} W} \right] , 
\nonumber
\\
\dim_{\textrm{r}}(X) 
& =
(-1)^{\binom{n+1}{2}}\Res \left[ \frac{ \str\big( \partial_{x_1} d_{X}\ldots \partial_{x_m} d_{X} \,  \partial_{y_1} d_{X}\ldots \partial_{y_n} d_{X}\big) \operatorname{d}\! x}{\partial_{x_1} V, \ldots, \partial_{x_m} V} \right] .
\end{align}
These formulas arose in the study of adjunctions in the bicategory of Landau-Ginzburg models, see \cite{cr1006.5609, cm1208.1481} and Remark~\ref{rem:intro} below. Here we just point out the following properties:
\begin{itemize}
\item[-] The definition of the quantum dimensions is independent of the $\Q$-grading on~$X$. 
In fact, the same residue expressions also work for ungraded matrix factorisations; in this case however $\dim_{\textrm{l}}(X)$ and $\dim_{\textrm{r}}(X)$ are not necessarily numbers, but polynomials in $k[x]$ and $k[y]$, respectively, see Lemma~\ref{lem:dimink}. 
\item[-] The quantum dimensions only depend on the isomorphism class of~$X$ in $\hmfgr(k[x,y],W-V)$ or $\hmf(k[x,y],W-V)$, cf.~Sections~\ref{subsec:pivotalbicat} and~\ref{subsec:MF}. Hence they are independent of the choice of isomorphism between~$X$ and $k[x,y]^r \oplus k[x,y]^r$.
\item[-] The quantum dimensions are multiplicative for a suitably defined tensor product. In addition, there is a duality $(-)^\dagger$ on matrix factorisations which exchanges left and right quantum dimensions 
(see again Sections~\ref{subsec:pivotalbicat} and~\ref{subsec:MF}).
\end{itemize}

\begin{defthm}\label{defthm:MF-orbequiv}
We say that two potentials $V(x)$, $W(y) \in \Pc_k$ are \textsl{orbifold equivalent}, $V \sim W$, if there exists a finite-rank graded matrix factorisation of $W-V$ for which the left and right quantum dimension are non-zero. This defines an equivalence relation on $\Pc_k$.
\end{defthm}

We shall prove the above statement, as well as Propositions~\ref{prop:sum+Kn"orrer} and~\ref{prop:MF-summand}, in Section~\ref{subsec:MF}. 
The name `orbifold equivalence' has its roots in the study of orbifolds via defects in two-dimensional quantum field theories \cite{ffrs0909.5013, dkr1107.0495, cr1210.6363}, see also Remark~\ref{rem:intro}. 

\medskip

Two basic properties of orbifold equivalences are:

\begin{proposition}\label{prop:sum+Kn"orrer}
Let $U(z),V(x),W(y) \in \Pc_k$.
\begin{enumerate}
\item (\textsl{Compatibility with external sums}) If $V \sim W$, then $V+U \sim W+U$. 
\item (\textsl{Kn\"orrer periodicity}) $V \sim V + u^2 + v^2$.
\end{enumerate}
\end{proposition}

To a potential $V \in \Pc_k$ one assigns a rational number, its \textsl{central charge} $c(V)$. Namely, for $V \in k[x_1,\dots,x_m]$ we have 
\be\label{eq:def-Vir-central-charge}
	c(V) = 3\sum_{i=1}^m \big(1-|x_i|\big) \, .
\ee
For example, if $V = x_1^{\,d} + x_1  x_2^{\,2}$, then $|x_1| = \tfrac2d$, $|x_2|=1-\tfrac1d$ and $c(V) = 3-\frac 3 d$, which is the same as the central charge of $W = y_1^{\,2d}+ y_2^{\,2}$. The following proposition gives some simple necessary conditions for any two potentials to be orbifold equivalent.

\begin{proposition}\label{prop:necessary}
Suppose $V,W \in \Pc_k$ are orbifold equivalent. Then
\begin{enumerate}
\item $m - n$ is even, where $V \in k[x_1,\dots,x_m]$ and $W \in k[y_1,\dots,y_n]$.
\item $c(V) = c(W)$.
\end{enumerate}
\end{proposition}

Part (i) is trivial since the supertrace in~\eqref{eq:MF-q-dim} is zero for an odd matrix. Part (ii) is proved in \cite[Sect.\,6.2]{cr1210.6363}. 

We expect the converse of Proposition~\ref{prop:necessary} to be false. For example, consider the family of potentials $x_1^{\,3} + x_2^{\,3} + x_3^{\,3} + c \, x_1 x_2 x_3$ with $c \in \C$ which all have central charge~3 and whose zero locus is an elliptic curve in $\C\PP^2$. In analogy with \cite{fgrs0705.3129, bbr1205.4647} we expect the potentials for different values of~$c$ to be orbifold equivalent iff the complex structure parameters of the corresponding curves are related by some $\operatorname{GL}(2,\Q)$ transformation, resulting in infinitely many equivalence classes for these potentials.

\medskip

If $V,W \in \Pc_k$ are orbifold equivalent, then the corresponding categories of matrix factorisations are closely related. Namely, let $X = (X,d_X)$ be a finite-rank graded matrix factorisation of $W(y) - V(x)$ with non-zero quantum dimensions. Let further $M = (M,d_M)$ be a graded matrix factoristiation of $V(x)$. Then their tensor product $X \otimes M := (X \otimes_{k[x]} M , d_X \otimes 1 + 1 \otimes d_M)$ is a graded matrix factorisation of $W(y)$ (which is necessarily of infinite rank, but still equivalent to a finite-rank factorisation, for $M \neq 0$). 

\begin{proposition}\label{prop:MF-summand}
Suppose that $V(x)$, $W(y) \in \Pc_k$ are orbifold equivalent and let $X \in \hmfgr(k[x,y],W-V)$ have non-zero quantum dimensions. Then every matrix factorisation in $\hmfgr(k[y],W)$ occurs as a direct summand of $X \otimes M$ for some $M \in \hmfgr(k[x],V)$.
This remains true if `$\hmfgr$' is replaced by `$\hmf$' everywhere. 
\end{proposition}

\subsection{Orbifold equivalence for simple singularities}\label{subsec:orbeqSimSin}

Now we take $k=\C$ and consider the subset of potentials which define simple singularities.\footnote{%
As we shall see in Section~\ref{subsec:simple-sing}, particularly in Remark~\ref{rem:galE6}, we can also work over the cyclotomic field $k = \Q(\zeta)$ for an appropriate root of unity~$\zeta$.} 
These fall into an ADE classification and can be taken to be the following elements of $\C[x_1,x_2]$ (see e.\,g.~\cite[Prop.\,8.5]{Yoshinobook}):
\begin{align}
V^{(\mathrm{A}_{d-1})} &= x_1^{\,d} + x_2^{\,2} & c&=3-3\cdot \tfrac{2}{d} &(d \geqslant 2)&
\nonumber\\
V^{(\mathrm{D}_{d+1})} &=  x_1^{\,d} + x_1 x_2^{\,2} & c&=3-3\cdot \tfrac{2}{2d} &(d \geqslant 3)&
\nonumber\\
V^{(\mathrm{E}_6)} &= x_1^{\,3} + x_2^{\,4} & c&=3-3\cdot \tfrac{2}{12}
\label{eq:simple-sing-ADE}
\\
V^{(\mathrm{E}_7)} &= x_1^{\,3} + x_1 x_2^{\,3} & c&=3-3\cdot \tfrac{2}{18}
\nonumber\\
V^{(\mathrm{E}_8)} &= x_1^{\,3} + x_2^{\,5} & c&=3-3\cdot \tfrac{2}{30}
\nonumber
\end{align}
From Proposition~\ref{prop:necessary} we know that for two potentials to be orbifold equivalent, their central charges have to agree. The preimages of the central charge function on the potentials in~\eqref{eq:simple-sing-ADE} are precisely the sets
\begin{align}
&\big\{ V^{(\mathrm{A}_{d-1})} \big\} \quad  \text{for $d$ odd},
\nonumber \\
&\big\{ V^{(\mathrm{A}_{d-1})} , V^{(\mathrm{D}_{d/2+1})}\big\} \quad \text{for $d$ even and $d \notin \big\{12,18,30\big\}$},
\label{eq:equiv-classes}\\
&\big\{ V^{(\mathrm{A}_{11})},V^{(\mathrm{D}_{7})},V^{(\mathrm{E}_6)} \big\} \, , \quad \big\{ V^{(\mathrm{A}_{17})},V^{(\mathrm{D}_{10})},V^{(\mathrm{E}_7)} \big\} \, , \quad \big\{ V^{(\mathrm{A}_{29})},V^{(\mathrm{D}_{16})},V^{(\mathrm{E}_8)} \big\} \, .
\nonumber 
\end{align}
In fact, this list exhausts the relevant orbifold equivalences: 

\begin{theorem}\label{thm:ADEorbifolds}
The orbifold equivalence classes of the potentials~\eqref{eq:simple-sing-ADE} are precisely those listed in~\eqref{eq:equiv-classes}.
\end{theorem}

This is our main result (together with Corollaries~\ref{cor:hmfmod} and~\ref{cor:Etypemod} below), which we prove in Section~\ref{subsec:simple-sing} by explicitly constructing graded matrix factorisations~$X$ with non-zero quantum dimensions for the equivalences
$V^{(\mathrm{A}_{2d-1})} \sim  V^{(\mathrm{D}_{d+1})}$ (already given in \cite{cr1210.6363}), as well as $V^{(\mathrm{A}_{11})} \sim V^{(\mathrm{E}_6)}$, $V^{(\mathrm{A}_{17})} \sim V^{(\mathrm{E}_7)}$ and $V^{(\mathrm{A}_{29})} \sim V^{(\mathrm{E}_8)}$.

\medskip

Proposition~\ref{prop:MF-summand} and Theorem~\ref{thm:ADEorbifolds} can be strengthened to equivalences of categories by invoking the general theory of equivariant completion of \cite{cr1210.6363}: 

\begin{corollary}\label{cor:hmfmod}
For $V,W$ and~$X$ as in Proposition~\ref{prop:MF-summand} we have 
\be\label{eq:hmf=mod}
\hmfgr (k[y],W ) \cong \modu ( X^\dagger \otimes X ) \, . 
\ee
\end{corollary}
Here $\modu ( X^\dagger \otimes X )$ is the \textsl{category of modules over $X^\dagger \otimes X$}, which is made up of matrix factorisations of~$V$ together with a compatible action of the monoid $X^\dagger \otimes X \in \hmfgr(k[x,x'], V(x)-V(x'))$. For more details we refer to \cite{cr1210.6363}, or to \cite{BCP2, OberwolfachAbstract} for much shorter reviews. 

In Section~\ref{subsec:simple-sing} we will explicitly compute $X^\dagger \otimes X$ for the matrix factorisations~$X$ giving rise to the orbifold equivalences of simple singularities. For those involving E-type singularities we find that $X^\dagger \otimes X$ decomposes into sums of well-known matrix factorisations: 

\begin{corollary}\label{cor:Etypemod}
We have 
\begin{align}
\hmfgr\! \big( \C[x], V^{(\mathrm{E}_6)} \big) 
& \cong 
\modu\! \big( P_{\{0\}} \oplus P_{\{-3,-2,\ldots,3\}} \big) \, , 
\nonumber
\\
\hmfgr\! \big( \C[x], V^{(\mathrm{E}_7)} \big) 
& \cong 
\modu\! \big( P_{\{0\}} \oplus P_{\{-4,-3,\ldots,4\}} \oplus P_{\{-8,-7,\ldots,8\}} \big) \, , 
\label{eq:hmfEmod}
\\
\hmfgr\! \big( \C[x], V^{(\mathrm{E}_8)} \big) 
& \cong 
\modu\! \big( P_{\{0\}} \oplus P_{\{-5,-4,\ldots,5\}} \oplus P_{\{-9,-8,\ldots,9\}} \oplus P_{\{-14,-13,\ldots,14\}} \big) 
\nonumber
\end{align}
where the rank-one matrix factorisations $P_S$ of \cite{br0707.0922} are defined in~\eqref{eq:P_S-def}. 
\end{corollary}

\begin{remark}\label{rem:intro} 
\begin{enumerate}
\item 
The construction of the equivalence relation can be repeated in a number of slightly modified settings. For example, one can work with ungraded matrix factorisations, or one can work over any commutative ring~$k$ as in \cite{cm1208.1481}. Or, instead of using a $\Q$-grading, one can consider matrix factorisations with R-charge in the sense of \cite{cr1006.5609} (which is more general). 

The setting used in this paper was chosen to be on the one hand as simple as possible -- hence working over $k\subset \C$ -- and on the other hand to be strong enough for us to be able to prove the decomposition of simple singularities into equivalence classes \eqref{eq:equiv-classes} -- hence the homogeneous potentials and the $\Q$-grading. 
(In the ungraded setting we do not know how to exclude the existence of equivalences beyond those in~\eqref{eq:equiv-classes}.)
\item
The decompositions~\eqref{eq:equiv-classes} and~\eqref{eq:hmfEmod} are expected from two-dimensional rational conformal field theory. According to the (conjectural) CFT/LG correspondence, the infrared fixed point of the Landau-Ginzburg model with potential~$V$ is a conformal field theory with central charge $c(V)$ \cite{m1989,vw1989,howewest}. 
Mathematically, this predicts a relation between the graded matrix factorisations of a given potential and the representation theory of a super vertex operator algebra. This relation is supported by explicit computations on both sides in examples but not understood (or even precisely formulated) in general.

For simple singularities, the fixed points are $\mathcal N=2$ supersymmetric minimal models. The classification of the latter contains the above ADE series \cite{m1989,vw1989, cv9211097, g9608063, g0812.1318}. Matrix factorisations of the difference of two potentials are interpreted as line defects on the CFT side \cite{br0707.0922, br0712.0188, cr0909.4381, cr1006.5609}. It is known in rational CFT \cite{Frohlich:2006ch, ffrs0909.5013} that there are topological line defects with non-zero quantum dimensions, linking the minimal models in the equivalence classes corresponding to~\eqref{eq:equiv-classes}.
We will compare this with our results in more detail in Section~\ref{sec:CFTcomparison}.

\item
It was shown in \cite{kst0511155} that for a simple singularity $V\in \C[x]$, $\hmfgr(\C[x],V)$ is equivalent to $\DD^{\textrm{b}}(\operatorname{Rep}\C Q)$, where~$Q$ is (any choice of) the associated Dynkin quiver. Thus by Theorem~\ref{thm:ADEorbifolds} and \cite{cr1210.6363} the derived representation theory of ADE quivers enjoys orbifold equivalences analogous to~\eqref{eq:hmfEmod}.\footnote{%
Such a relation between A- and D-type quivers was already proven in \cite{ReitenRiedtmann} by different methods.} 
The monoids $X^\dagger \otimes X \cong P_{\{0\}} \oplus \ldots$ translate into functors on $\DD^{\textrm{b}}(\operatorname{Rep}\C Q)$ whose actions on simple objects are easily computable. 
\end{enumerate}
\end{remark}

\subsubsection*{Acknowledgements}

We thank 
    Hanno Becker, 
    Ilka Brunner, 
    Calin Lazaroiu,
    Wolfgang Lerche,
    Daniel Murfet, 
    Daniel Plencner, 
    Alexander Polishchuk, 
    Daniel Roggenkamp,
    Eric Sharpe,
    Duco van Straten, 
    Catharina Stroppel,
and 
    Atsushi Takahashi. 
In addition we are grateful to the organisers of the inspirational Oberwolfach workshop ``Matrix Factorizations in Algebra, Geometry, and Physics'', without which numerous helpful discussions would not have taken place.
We also~acknowledge support from the German Science Foundation (DFG), N.\,C.~and I.\,R.~within the Collaborative Research Center 676 ``Particles, Strings and the Early Universe'', A.\,R.\,C.~as part of the Research Training Group 1670 ``Mathematics Inspired by String Theory and QFT''.

\section{Proofs}\label{sec:proofs}

In this section we provide proofs of the results summarised above, with the claims of Sections~\ref{subsec:introductionpart1} and~\ref{subsec:orbeqSimSin} proven in Sections~\ref{subsec:MF} and~\ref{subsec:simple-sing}, respectively. 
The discussion for simple singularities can also be viewed as showcasing methods that may prove useful for constructing further orbifold equivalences between potentials. 

The notion of orbifold equivalence actually makes sense between objects in any pivotal bicategory. For context and completeness this is explained in Section~\ref{subsec:pivotalbicat}, after which we turn to concrete matrix factorisations.

\subsection{Pivotal bicategories}\label{subsec:pivotalbicat}

In this section (only) we assume some familiarity with bicategories, referring to \cite{bor94} for an introduction. We shall adopt the notation and conventions of \cite{cr1210.6363}; more detailed explanations of the properties of duals and dimensions stated below can be found for example in \cite[Sect.\,2]{cr1006.5609} and \cite[Sect.\,2.1]{cm1208.1481}. 

We denote the unit 1-morphism of an object~$a$ in a bicategory~$\B$ by~$I_a$, its left action by~$\lambda$, and the (horizontal) composition of 1-morphisms $X \in \B(a,b)$, $Y\in \B(b,c)$ is written $Y\otimes X \in \B(a,c)$. 

We say that~$\B$ \textsl{has left adjoints} if for each $X \in \B(a,b)$ there is an $X^\dagger \in \B(b,a)$ together with 2-morphisms 
\be\label{eq:evcoev}
\ev_X : X^\dagger \otimes X \lra I_a
\, , \quad 
\coev_X : I_b \lra X \otimes X^\dagger
\ee
which are the counit and unit of an adjunction, see \cite{GMbook}. Such a~$\B$ is \textsl{pivotal} if for all $a,b\in\B$ there are natural isomorphisms $\delta^{a,b}$ between the identity functor and $(-)^{\dagger\dagger}$ restricted to $\B(a,b)$, compatible with~$\otimes$. In this case~$X^\dagger$ is also right adjoint to~$X$ as exhibited by the maps
\be\label{eq:evcoevtilde}
\tev_X : X \otimes X^\dagger \lra I_b
\, , \quad
\tcoev_X : I_a \lra X^\dagger \otimes X
\ee
defined by $\ev_{X^\dagger} \circ (\delta_X^{a,b} \otimes 1_{X^\dagger})$ and $(1_{X^\dagger} \otimes (\delta_X^{a,b})^{-1}) \circ \coev_{X^\dagger}$, respectively. 

From now on we assume that we are given a pivotal bicategory~$\B$. For every $X\in \B(a,b)$ its \textsl{left} and \textsl{right quantum dimensions} are defined as 
\be
\diml(X) = \ev_X \circ\, \tcoev_X \in \End(I_a)
\, , \quad 
\dimr(X) = \tev_X \circ \coev_X \in \End(I_b) \, .
\ee
One can show that the quantum dimensions only depend on the isomorphism class of~$X$, and that
\be\label{eq:dimadj}
\diml(X^\dagger) = \dimr(X) \, . 
\ee

The left quantum dimension of~$X$ is the image of $1_{I_b}$ under the map $\Dl(X) : \End(I_b) \rightarrow \End(I_a)$ given by 
\be
\Dl(X) = \ev_X \circ \, [1_{X^\dagger} \otimes (\lambda_X \circ ((-)\otimes 1_X) \circ \lambda_X^{-1})) ] \circ \, \tcoev_X  \,  ,
\ee 
and similarly $\dimr(X) = \Dr(X)(1_{I_a})$. One finds that for $Y\in \B(b,c)$ these operators satisfy 
\be\label{eq:dimmult}
\Dl(X) \circ \Dl(Y) = \Dl (Y\otimes X) 
\, , \quad 
\Dr(Y) \circ \Dr(X) = \Dr (Y\otimes X) \, . 
\ee

\begin{definition}\label{def:orb-equiv}
Let~$\B$ be a pivotal bicategory.
Two objects $a,b\in\B$ are \textsl{orbifold equivalent}, $a\sim b$, if there is an $X\in\B(a,b)$ such that
\begin{enumerate}
\item $X$ has invertible left and right quantum dimensions, and
\item $\Dl(X)$ and $\Dr(X)$ map invertible quantum dimensions to automorphisms. 
\end{enumerate}
\end{definition}

Condition (ii) is needed for orbifold equivalence in this general setting to be an equivalence relation (see Theorem~\ref{thm:orb-equiv}). 
This condition may be hard to check directly, as it requires knowing all invertible elements in $\End(I_{a})$ and $\End(I_{b})$ which arise as quantum dimension of some 1-morphism. 
A special case where condition (ii) is already implied by condition (i) is if~$\B$ is such that every Hom-category $\B(a,b)$ is $R$-linear for some fixed commutative ring~$R$, and such that the left and right quantum dimensions of every 1-morphism $X \in \B(a,b)$ are in $R\cdot 1_{I_a}$ and $R\cdot 1_{I_b}$, respectively. In this case one may think of the left/right quantum dimension as an element of $R$, and one can check that $\dim_{\mathrm{l/r}}(X \otimes Y) = \dim_{\mathrm{l/r}}(X) \cdot \dim_{\mathrm{l/r}}(Y)$, where the product on the right is in $R$. Graded finite-rank matrix factorisations are an example of this $R$-linear setting (see Lemma~\ref{lem:dimink} below).

\begin{theorem}\label{thm:orb-equiv}
Orbifold equivalence is an equivalence relation. 
\end{theorem}

\begin{proof}
To see reflexivity $a\sim a$ one can take $X = I_a$, for which $\Dl(I_a) = 1_{I_a} = \Dr(I_a)$. 

For symmetry and transitivity it is helpful to expand condition (ii) of the equivalence relation in more detail: 
\begin{enumerate}
\item[(ii-r)] 
for every $e\in\B$ and every $Z \in \B(e,a)$ such that $\dimr(Z)$ is invertible in $\End(I_a)$, we have that $\Dr(X) (\dimr(Z))$ is invertible in $\End(I_b)$; 
\item[(ii-l)] 
for every $f\in\B$ and every $Z' \in \B(b,f)$ such that $\diml(Z')$ is invertible in $\End(I_b)$, we have that $\Dl(X) (\diml(Z'))$ is invertible in $\End(I_a)$.
\end{enumerate}
Of course, since (ii-r) holds for all $Z$, it also holds for all $Z^\dagger$, and hence $\Dr(X)$ also maps invertible left quantum dimensions to automorphisms. The same applies to (ii-l).

Symmetry $a\sim b \Leftrightarrow b\sim a$ follows by replacing $X$ by $X^\dagger$. That $X^\dagger$ satisfies (i) follows from \eqref{eq:dimadj}. To see (ii-r) for $X^\dagger$, use $\Dr(X^\dagger)=\Dl(X)$ and that (ii-l) holds for $X$. Condition (ii-l) follows analogously.

To check transitivity, consider $X\in\B(a,b)$ and $Y\in\B(b,c)$ satisfying conditions (i) and (ii). Then $\dimr(Y \otimes X) = \Dr(Y) (\dimr(X))$ which is invertible by~(i) for~$X$ and (ii-r) for $Y$. Dito for $\diml$. Next let $Z \in \B(e,a)$ be as in (ii-r) above. Then $\Dr(Y \otimes X) (\dimr(Z)) = \Dr(Y) (\dimr(X \otimes Z))$. Now $\dimr(X \otimes Z)$ is invertible by (ii-r) for $X$ and therefore $\Dr(Y) (\dimr(X \otimes Z))$ is invertible by (ii-r) for $Y$. That $Y \otimes X$ satisfies (ii-l) is seen similarly.
\end{proof}

We conclude our brief general discussion with an implication of the existence of orbifold equivalences in the setting of retracts and idempotent splittings. 
An object~$S$ in some category is called a \textsl{retract} of an object~$U$ if there are morphisms $e : S \to U$ and $r : U \to S$ such that $r \circ e = 1_S$. In particular, $e$ is mono and so~$S$ is a subobject of~$U$.

\begin{proposition}\label{prop:all-is-retract}
Let $a,b,c,d \in \B$ and let $X \in \B(a,b)$, $Y \in \B(c,d)$ have invertible left quantum dimensions. 
Every~$Z \in \B(a,c)$ is a retract of $Y^\dagger \otimes F \otimes X$ for some suitable ($Z$-dependent) $F \in \B(b,d)$.
\end{proposition}

\begin{proof}
We set $F = Y \otimes Z \otimes X^\dagger$ and define the maps $e : Z \to Y^\dagger \otimes F \otimes X$ and $r : Y^\dagger \otimes F \otimes X \to Z$  by
\begin{align}
e &= \tcoev_Y \otimes 1_Z \otimes \tcoev_X \, ,
\nonumber \\
r &= \big(\diml(Y)^{-1} \otimes 1_Z \otimes \diml(X)^{-1}\big) \circ \big(\ev_Y \otimes 1_Z \otimes \ev_X \big) \, .
\end{align}
Clearly, $r \circ e = 1_Z$.
\end{proof}

\begin{remark}\label{rem:all-direct-summand}
If the Hom-categories of~$\B$ are additive and if idempotent 2-morphisms split, we can improve on Proposition~\ref{prop:all-is-retract} slightly: instead of~$Z$ just being a retract, it now even occurs as a direct summand of $Y^\dagger \otimes F \otimes X$ for a suitable~$F$. To see this, take the maps $e,r$ from the proof, note that $p = e \circ r$ is an idempotent endomorphism of $Y^\dagger \otimes Y \otimes Z \otimes X^\dagger \otimes X$, and consider the decomposition $1 = p + (1-p)$ into orthogonal idempotents.
\end{remark}

\subsection{Matrix factorisations}\label{subsec:MF}

Potentials and (graded) matrix factorisations as defined in Section~\ref{sec:introduction} form bicategories that are `as good as' pivotal,\footnote{One has pivotality `on the nose' if one restricts to potentials in an even number of variables, as we do in our discussion of simple singularities in Section~\ref{subsec:simple-sing}.} 
as explained in detail in \cite{cm1208.1481}. Accordingly we could present the results of this section as applications of a slight variant of the previous one. We will however take a different route and give more direct arguments. 

\medskip

Recall from \cite{cr1006.5609, bfk1105.3177, cm1208.1481} that every (ungraded or graded) matrix factorisation~$X$ of $W(y_1,\ldots,y_n) - V(x_1,\ldots,x_m)$ has left and right adjoints. If we write $X^\vee = \Hom_{k[x,y]}(X,k[x,y])$ and set $d_{X^\vee}(\nu) = (-1)^{|\nu|+1} \nu \circ d_X$ for $\Z_2$-homogeneous $\nu \in X^\vee$, then they are given by $X^\vee[n]$ and $X^\vee[m]$, respectively. These are (ungraded or graded) matrix factorisations of $V-W$, and the associated adjunction maps as in~\eqref{eq:evcoev}, \eqref{eq:evcoevtilde} are known explicitly, see \cite{cm1208.1481} for details or \cite{cm1303.1389} for a concise review. 

Control over adjunctions leads to the explicit expressions for the quantum dimensions given in~\eqref{eq:MF-q-dim}. 
For ungraded matrix factorisations these are polynomials, but in the graded case the quantum dimensions are just numbers: 

\begin{lemma}\label{lem:dimink}
Quantum dimensions of graded matrix factorisations take value in~$k$. 
\end{lemma}

\begin{proof}
By construction the adjunction maps of \cite{cm1208.1481} have degree zero for any graded matrix factorisation~$X$. Hence $\diml(X)$ and $\dimr(X)$ must be in~$k$ since the variables have positive degrees and every closed endomorphism of the unit factorisation is homotopy equivalent to a polynomial. Alternatively, one may simply count degrees in the explicit formulas in~\eqref{eq:MF-q-dim}. 
\end{proof}

As a consequence of this lemma, in the graded case the statements ``$X$ has non-zero quantum dimensions'' and ``$X$ has invertible quantum dimensions'' are equivalent. 

\medskip

We may now proceed to prove the results advertised in Section~\ref{subsec:introductionpart1}. 

\begin{proof}[Proof of Theorem~\ref{defthm:MF-orbequiv}]
This is a corollary of \cite[Prop.\,8.5]{cm1208.1481}: reflexivity and symmetry follow analogously to Theorem~\ref{thm:orb-equiv}, and transitivity follows from the fact that the quantum dimensions~\eqref{eq:MF-q-dim} are manifestly multiplicative up to a sign. 
\end{proof}

\begin{proof}[Proof of Proposition~\ref{prop:sum+Kn"orrer}]
From~\eqref{eq:MF-q-dim} it is clear that quantum dimensions are multiplicative (up to a sign) also for the external tensor product $\otimes_k$, showing part~(i). 
Part~(ii) is proven in \cite[Sect.\,7.2]{cr1210.6363}. 
\end{proof}

\begin{proof}[Proof of Proposition~\ref{prop:MF-summand}]
This is a direct consequence of the proof of Proposition~\ref{prop:all-is-retract} and Remark~\ref{rem:all-direct-summand} as we are dealing with idempotent complete matrix factorisation categories. 
\end{proof}

\subsection{Simple singularities}\label{subsec:simple-sing}

Let $W(y_1,y_2)$ be one of the potentials $V^{(\mathrm{D}_{d/2+1})}, V^{(\mathrm{E}_6)}, V^{(\mathrm{E}_7)}, V^{(\mathrm{E}_8)}$ in~\eqref{eq:simple-sing-ADE}, and let $V(x_1,x_2) = V^{(\mathrm{A}_{d-1})} = x_1^d + x_2^2$ be the corresponding A-type potential of the same central charge as~$W$. To avoid too many indices, we  rename
\be
	x_1 \leadsto u \, , \quad
	x_2 \leadsto v \, , \quad
	y_1 \leadsto x \, , \quad
	y_2 \leadsto y \, .
\ee

In this section we will give a finite-rank graded matrix factorisation of $W(x,y)-V(u,v)$ of non-zero left and right quantum dimension in each case, thus proving the results collected in Section~\ref{subsec:orbeqSimSin}. These matrix factorisations have all been constructed along the following lines:
\begin{enumerate}
\item Pick a matrix factorisation $X_0$ of $W(x,y) - v^2$ that is of low rank.
\item Thinking of~$u$ as a deformation parameter, add to each entry of the matrix $d_{X_0}$ the most general homogeneous polynomial of the form $u \cdot p(u,v,x,y)$ with the same total degree as the given entry. Let $d_{X_u}$ be the resulting matrix. 
\item Reduce the number of free parameters in the polynomials~$p$ by absorbing some of them via a similarity transformation $d_{X_u} \mapsto \Phi \circ d_{X_u} \circ \Phi^{-1}$.
\item Try to find a set of parameters (the remaining coefficients in the polynomials~$p$) such that $d_{X_u} \circ d_{X_u} = (W-V) \cdot 1$.
\end{enumerate}
Choosing a low-rank starting point in~(i) and reducing the number of parameters via (iii) only serves to simplify the problem in~(iv). We will work through these steps for $W=V^{(\mathrm{D}_{d/2+1})}$ and $W=V^{(\mathrm{E}_6)}$ in some detail, while our discussion will be briefer for $V^{(\mathrm{E}_7)}$ and $V^{(\mathrm{E}_8)}$.

\subsubsection*{$\boldsymbol{V^{(\mathrm{D}_{d/2+1})} \sim V^{(\mathrm{A}_{d-1})}}$}

We set $b=d/2$, so that $W = x^b + x y^2$ and $V = u^{2b}+v^2$, and we write 
$R = \C[u,v,x,y]$. 
As the starting point in step (i) we choose $X_0 = (X , d_{X_0})$ with $\Z_2$-graded $R$-module $X = R^2 \oplus R^2$ and twisted differential
\be\label{eq:D-calc-X0}
	d_{X_0} = \begin{pmatrix} 0 & d_{X_0}^1 \\ d_{X_0}^0 & 0 \end{pmatrix}
	\quad \text{with} \quad
	d_{X_0}^1 = \begin{pmatrix} x & v \\ v & x^{b-1} + y^2 \end{pmatrix}
	.
\ee
Since $\det(d_{X_0}^1) = x^b + x y^2 - v^2$, the component $d_{X_0}^0$ is determined to be the adjunct matrix, $d_{X_0}^0  = (d_{X_0}^1)^\#$. Recall that the adjunct $M^\#$ of an invertible matrix~$M$ has entries which are polynomial in those of~$M$ and satisfies $M^{-1} = (\det M)^{-1} \cdot M^\#$.

The deformed matrix factorisation $X_u$ has the same underlying $R$-module $R^2 \oplus R^2$. 
For step (ii) we need to pick the most general homogeneous deformation of $d_{X_0}^1$, which is
\be\label{eq:D-calc-ansatz}
	d_{X_u}^1 = \begin{pmatrix} 
	x + u \, p_{11} & v  + u \, p_{12}  \\ v  + u \, p_{21}  & x^{b-1} + y^2  + u \, p_{22} 
	\end{pmatrix}
	\quad \text{where} \quad p_{ij} \in \C[u,v,x,y] \, .
\ee
The degrees of the variables are $|u|=1/b$, $|v|=1$, $|x|=2/b$, and $|y| = 1-1/b$. Hence the total degrees of the $p_{ij}$ have to be
\be
	|p_{11}| = \tfrac1b	\, , \quad
	|p_{12}| = 1-\tfrac1b	\, , \quad
	|p_{21}| = 1-\tfrac1b	\, , \quad
	|p_{22}| = 2-\tfrac3b	 \, .
\ee
From this it follows that $p_{11} = a_1 u$ for some $a_1 \in \C$, but the remaining $p_{ij}$ will contain of the order~$b$ many free coefficients. 

Moving to step (iii), we will now use degree-preserving row and column operations on $d_{X_u}^1$ to reduce the number of free coefficients. By applying the inverse operations to $d_{X_u}^0$ one produces in this way an isomorphic matrix factorisation. The most general such row and column manipulations turn out to be
\be\label{eq:D-calc-ansatz-reduced}
	\begin{pmatrix} 1 & 0 \\ f & 1 \end{pmatrix} d_{X_u}^1 \begin{pmatrix} 1 & g \\ 0 & 1 \end{pmatrix}
	= \begin{pmatrix}
	x + a_1 u^2 & v  + u \, p_{12}  + g (x + a_1 u^2)  \\ v  + u \, p_{21}  + f (x + a_1 u^2) & *
	\end{pmatrix}
\ee
where $|f|=|g|=1-2/b$. We see that $f,g$ can be used to remove any $x$-dependence from $p_{12}$ and $p_{21}$, and we arrive at the following reduced ansatz: $d_{X_u}^1$ of the form \eqref{eq:D-calc-ansatz} with
\begin{align}
	p_{11} &= a_1 u \, , &
	p_{12} &= a_2 y + a_3 u^{b-1} \, , \nonumber \\
	p_{21} &= a_4 y + a_5 u^{b-1} \, ,  &
	p_{22} &= q_1  v + q_2  y + q_3  \, ,
\end{align}
where $a_i \in \C$ and $q_i \in \C[u,x]$ with degrees $|q_1| = 1-3/b$, $|q_2| = 1-2/b$, and $|q_3|=2-3/b$. 

Step (iv) amounts to the tedious task of trying to find conditions such that $d_{X_u} \circ d_{X_u} = (W-V) \cdot 1$.
The second component of the twisted differential is uniquely determined to be $d_{X_u}^0  = q / \det(d_{X_u}^1) \cdot (d_{X_u}^1)^\#$ with $q=x^b + x y^2 - u^{2b} - v^2$, and we need to find values of the deformation parameters so that this matrix has polynomial entries. Since $q \in \C[x,y,u,v]$ is irreducible, either $q$ is a factor of $\det(d_{X_u}^1)$, or $\det(d_{X_u}^1)$ has to cancel against the entries of $(d_{X_u}^1)^\#$. Degree considerations show that the latter is not possible, and in fact $\det(d_{X_u}^1)$ equals~$q$ up to a multiplicative constant. By rescaling $d_{X_u}^1$ if necessary, without restriction of generality we can impose $\det(d_{X_u}^1) = q$.
In solving this condition, one is lead to distinguish between two cases, $a_2=0$ and $a_2 \neq 0$. Setting $a_2=0$ produces a solution with zero left and right quantum dimension. On the other hand, keeping $s := a_2 \neq 0$ forces
\be \label{eq:D-calc-solution}
	d_{X_u}^1 = \begin{pmatrix} 
	x - (su)^2 & v  + y(s u)  \\ v  - y(s u)  & \frac{x^b - (su)^{2b}}{x - (su)^2} + y^2 
	\end{pmatrix}
	\quad \text{where} \quad s^{2b} = 1 \, .
\ee
For $s=1$ this is the solution given in \cite[Sect.\,7.3]{cr1210.6363}. The quantum dimensions are  
\be
	\diml(X_u) = -2 s
	\, , \quad
	\dimr(X_u) = - s^{-1} \, .
\ee

What remains to be done is to check that there exists a $\Q$-grading on $X_u = (X, d_{X_u})$ such that $d_{X_u}$ has $\Q$-degree 1. It is in fact easy to write down all such gradings. The $\Q$-grading on the ring~$R$ is fixed by $|1|=0$, with the variables $u,v,x,y$ having degrees as stated below \eqref{eq:D-calc-ansatz}. Writing $R[\alpha]$, $\alpha \in \Q$, for~$R$ with $|1|=\alpha$, the possible gradings on~$X$ are
\be
% degree calculation :
%	d1 :	2/b	 1
%		1	 2-2/b	
%	d0 :	2-2/b 1
%		1	 2/b
	X^0 = R[\alpha] \oplus R[\alpha-1+\tfrac2b] \, , \quad
	X^1 = R[\alpha-1+\tfrac2b] \oplus R[\alpha]
	\, , \quad \alpha \in \Q \, .
\ee
Equivalently, the grading data on~$X$ may be presented as a `$U$-matrix',  cf.~e.\,g.~\cite{cr1006.5609}, 
\be
U_{X}(\beta) = 
\E^{\I \alpha \beta} \cdot 
\operatorname{diag} \left( 1, \, \E^{(2/b - 1)\I \beta}, \, \E^{(2/b - 1)\I \beta} , \, 1\right)
.
\ee

This completes the proof that the potentials $V^{(\mathrm{A}_{d-1})}$ and $V^{(\mathrm{D}_{d/2+1})}$ are orbifold equivalent. We can summarise the above discussion as follows.

\begin{lemma}\label{lem:D-calc-summary}
Let~$Y$ be a rank-two $\Q$-graded matrix factorisation of $x^b + x y^2 - u^{2b} - v^2$. Suppose that~(i) $Y$ has non-zero quantum dimensions and~(ii) when setting~$u$ to zero, $Y$ is equal to $X_0$ in~\eqref{eq:D-calc-X0}. Then~$Y$ is isomorphic in $\hmfgr(\C[u,v,x,y],x^b + x y^2 - u^{2b} - v^2)$ to $X_u$ in~\eqref{eq:D-calc-solution} for some choice of~$s$.
\end{lemma}

Let us analyse the matrix factorisation $X_u$ a bit further. 
First we recall the definition of the permutation matrix factorisations~$P_S$ from \cite{br0707.0922}, which are factorisations of $u'^{\,d} - u^{\,d}$. For a subset $S \subset \Z_d$ we denote by~$S^{\textrm{c}}$ its complement in~$\Z_d$. Then one has $P_S = (\C[u,u'] \oplus \C[u,u'] , d_{P_S})$ with
\be\label{eq:P_S-def}
	d_{P_S}^1 = \prod_{l \in S} (u' - \zeta_d^l u) \, , \quad
	d_{P_S}^0 = \prod_{l \in S^{\textrm{c}}} (u' - \zeta_d^l u) \, , \quad \text{where} \quad
	\zeta_d = \E^{2 \pi \I/d} 
\ee
and the dependence on~$d$ should be clear from the context. 
The left and right quantum dimensions were computed in \cite[Sect.\,3.3]{cr1006.5609} to be
\be
  \diml(P_S) = \sum_{l \in S} \zeta_d^l 
  \, , \quad 
  \dimr(P_S) = \sum_{l \in S} \zeta_d^{-l} \, . 
\ee

The automorphism~$\sigma$ of $\C[u]$ determined by $\sigma(u) = \zeta_d u$ leaves the potential $u^d$ invariant. 
Given a matrix factorisation~$X$ of $W - u^d$ for some $W \in \C[x]$, twisting the $\C[u,x]$-action on~$X$ by $\sigma^l$ (with $\sigma$ extended to $\C[u,x]$ via $\sigma(x_i)=x_i$) results again in a matrix factorisation of $W - u^d$ which we denote by $X_{\sigma^l}$. It follows from \cite[Lem.\,2.10]{cr0909.4381} that
\be\label{eq:P_{-k}-action}
	X_{\sigma^l} \cong X \otimes P_{\{-l\}} 
\ee
in $\hmfgr(\C[u,x], W-u^d)$. 

\begin{remark}\label{rem:D-calc-s-choice}
The observation \eqref{eq:P_{-k}-action} together with the classification given in Lemma~\ref{lem:D-calc-summary} explains the $2b$-th root of unity~$s$ appearing as a parameter in \eqref{eq:D-calc-solution}. To wit, given any 1-morphism $P : V \to V$ of non-zero quantum dimensions, the composition $X_u \otimes P : V \to W$ also has non-zero quantum dimensions. Consider the choice $P = P_{\{l\}} \otimes_{\C} I_{v^2}$, where $I_{v^2}$ the unit 1-morphism for the potential~$v^2$ with twisted differential $(\begin{smallmatrix} 0 & v'-v \\ v'+v & 0 \end{smallmatrix})$. 
Then~$P$ is a 1-endomorphism of $u^{2b}+v^2$ and $X_u \otimes P$ is again a rank-two factorisation satisfying the conditions in Lemma~\ref{lem:D-calc-summary}. Hence $X_u \otimes P$ must be isomorphic to~$X_u$ for a possibly different choice of~$s$. But different choices of~$s$ precisely amount to twisting the $u$-action by some power of the automorphism~$\sigma$.
\end{remark}

It was checked in \cite[Sect.\,7.3]{cr1210.6363} that for $d\in \{ 2,3,\ldots,10 \}$, and up to a trivial factor of the unit $I_{v^2}$ the monoid $X^\dagger_{u'} \otimes X_u$ is isomorphic to
\be\label{eq:PPinDcase}
P_{\{0\}} \oplus P_{\{0,1,\ldots,d-1\} \backslash \{\frac{d}{2}\}} \, . 
\ee
It is straightforward to compute an explicit basis for the endomorphisms of~$X_u$ in the homotopy category, e.\,g.~by using the Singular code of \cite{cdr1112.3352}. We have done so for small values of~$d$, with the result 
\be\label{eq:EndXDcase}
\End
(X_u) 
\cong 
\Big(
\bigoplus_{i=0}^{d-2} \C_{i\cdot \frac{2}{d}}
\Big)
\oplus 
\C_{\frac{d-2}{d}}
\ee
where~$\C_{j}$ denotes the one-dimensional subspace of maps of $\Q$-degree~$j$. 
Both~\eqref{eq:PPinDcase} and~\eqref{eq:EndXDcase} are expected to be the correct expressions for all~$d$.

\subsubsection*{$\boldsymbol{V^{(\mathrm{E}_6)} \sim V^{(\mathrm{A}_{11})}}$}

Our starting point in step~(i) now is the matrix factorisation $X_0 = (R^2 \oplus R^2 , d_{X_0})$ of $x^3 + y^4 - v^2$ 
with $R = \C[u,v,x,y]$ and twisted differential
\be\label{eq:E6-calc-X0}
	d_{X_0}^1 = \begin{pmatrix} y^2-v & -x \\ x^2 & y^2+v \end{pmatrix}
	 , \quad
	d_{X_0}^0  = (d_{X_0}^1)^\#
	\, .
\ee
This is one of the six indecomposable objects of $\hmfgr(\C[v,x,y],V^{(\mathrm{E}_6)} - v^2)$ listed in \cite[Sect.\,5]{kst0511155}. The variable degrees are
\be
	|u| = \tfrac16
	\, , \quad
	|v| = \tfrac66
	\, , \quad
	|y| = \tfrac36
	\, , \quad
	|x| = \tfrac46 
	\, .
\ee

Carrying out steps~(ii) and (iii) leads to the possibility
\be
	d_{X_u}^1 = \begin{pmatrix} 
	y^2-v + a_1 x u^2 + a_2 u^6 & 
	-x + a_3 y u + a_4 u^4 \\ 
	x^2 + a_5 y x u +  a_6 x u^4 + a_7 v u^2 + q & 
	y^2+v + a_8 x u^2 + a_9 u^6
	\end{pmatrix} 
\ee
for the reduced ansatz, where $a_i \in \C$ and $q \in \C[u,y]$ with $|q| = 8/6$. Of course a different choice of similarity transformation may give a different (but isomorphic) reduced ansatz. Here, the row and column manipulations were used to absorb the terms $y u^3$ in the diagonal entries.

By the same argument as used in the D-case, we now need to solve 
the condition $\det(d_{X_u}^1)= x^3 + y^4 - u^{12} - v^2$ under the extra constraint that $X_u$ has non-zero quantum dimensions. This leads to 
\begin{align} \label{eq:E6-calc-solution}
d_{X_u}^1 = \big(\begin{smallmatrix} a&b \\ c & d \end{smallmatrix}\big) \quad \text{with} \quad 
	a &= y^2-v + \tfrac12 x (s u)^2 + \tfrac{2t+1}8  (s u)^6   \, ,
	\nonumber \\
	b &= -x + y (s u) + \tfrac{t+1}4  (s u)^4   \, ,
	\nonumber \\
	c &= x^2 + y x (s u) +   \tfrac{t}4 x (su)^4 + \tfrac{2t+1}4 y (su)^5
	 - \tfrac{9t+5}{48} (su)^8   \, ,
	\nonumber \\
	d &= y^2+v + \tfrac12 x (s u)^2 + \tfrac{2t+1}8 (s u)^6 \, ,
\end{align}
and $d_{X_0}^0  = (d_{X_0}^1)^\#$. Here, $s$ and~$t$ can be any solution of
\be\label{eq:stE6}
	t^2  = \tfrac13
	\, , \quad	
	s^{12} = - 576 \, (26 \, t -15) \, .
\ee
Note that~$s$ can be modified by a 12-th root of unity -- the interpretation of this is as in Remark~\ref{rem:D-calc-s-choice}.

The quantum dimensions of $X_u$ are
\be
	\diml(X_u) = s
	\, , \quad
	\dimr(X_u) = 3 \, (1 - t) s^{-1}
	 \, ,
\ee
and all possible $\Q$-gradings on $X_u$ are again easily found: the underlying $\Z_2$-graded $R$-module $X = X_0 \oplus X_1$ has components with $\Q$-grading
\be
% degree calculation :
%	d1 :	6/6	 4/6
%		8/6	 6/6	
%	d0 :	6/6	 4/6
%		8/6	 6/6	
	X^0 = R[\alpha] \oplus R[\alpha-\tfrac13] \, , \quad
	X^1 = R[\alpha] \oplus R[\alpha-\tfrac13]
	\, , \quad \alpha \in \Q \, ,
\ee
equivalently described by the $U$-matrix
\be
U_{X_u}(\beta) = 
\E^{\I \alpha \beta} \cdot 
\operatorname{diag} \left( 1, \, \E^{-\I \beta/3},\, 1,\, \E^{-\I \beta/3} \right)
.
\ee

As in the D-case, we can summarise the above as:

\begin{lemma}\label{lem:E6-calc-summary}
Let~$Y$ be a rank-two $\Q$-graded matrix factorisation of $x^3 + y^4 - u^{12} - v^2$. Suppose that (i)~$Y$ has non-zero quantum dimensions and (ii) when setting~$u$ to zero, $Y$ is equal to $X_0$ in~\eqref{eq:E6-calc-X0}. Then~$Y$ is isomorphic in $\hmfgr(\C[u,v,x,y],x^3 + y^4 - u^{12} - v^2)$ to $X_u$ in~\eqref{eq:E6-calc-solution} for some solution $s,t$ of~\eqref{eq:stE6}.
\end{lemma}

\medskip

Computing the endomorphism of~$X_u$ in the homotopy category one finds the 16-dimensional space 
\be\label{eq:EndhmfX-E6}
\End
(X_u) 
\cong 
\Big(
\bigoplus_{i=0}^{10} \C_{i\cdot \frac{2}{12}}
\Big)
\oplus 
\Big(
\bigoplus_{i=3}^{7} \C_{i\cdot \frac{2}{12}}
\Big) \, .
\ee

As mentioned in Section~\ref{subsec:orbeqSimSin} and explained in detail in \cite{cr1210.6363}, a matrix factorisation~$X$ of $W(y)-V(x)$ with invertible quantum dimensions allows us to describe \textsl{all} matrix factorisations of~$W$ in terms of modules over $A := X^\dagger \otimes X \in \hmfgr(k[x,x'], V(x)-V(x'))$. The matrix $d_{X^\dagger \otimes X}$ also depends on the $y$-variables, and hence the matrix factorisation~$A$ is of infinite rank. However, by the results of \cite{dm1102.2957} it is homotopy equivalent (and thus isomorphic in $\hmfgr(k[x,x'], V(x)-V(x'))$) to a finite-rank matrix factorisation. 

The construction of this finite-rank factorisation and the explicit homotopy equivalence can be implemented on a computer; this was done in \cite{khovhompaper}, where it was used to compute Khovanov-Rozansky link invariants. In our present situation we can use this implementation to find that $X_{u'}^\dagger \otimes X_u$ is equivalent to the matrix factorisation $A'\otimes_k I_{v^2}$, where the twisted differential of~$A'$ is represented by a 4-by-4 matrix~$a'$ with nonzero entries
\begin{align}
a'_{13}& = (\tfrac{76}{3}-44 t) u^8+(\tfrac{136}{3}-80 t) u' u^7+(44-76 t) u'^2 u^6+(\tfrac{152}{3}-88 t) u'^3 u^5 \nonumber \\ 
& \qquad +(\tfrac{208}{3}-120 t) u'^4 u^4+(\tfrac{152}{3}-88 t) u'^5 u^3+(20 t-12) u'^6 u^2 \nonumber \\ 
& \qquad +(88 t-\tfrac{152}{3}) u'^7 u+(52 t-\tfrac{92}{3}) u'^8 \nonumber \, , \\
a'_{14} & = (224-384 t) u^5 +(832-1440 t) u' u^4+(1440-2496 t) u'^2 u^3 \nonumber \\ 
& \qquad +(1440-2496 t) u'^3 u^2+(832-1440 t) u'^4 u+(224-384 t) u'^5 \nonumber \, , \\
a'_{23} & =  (\tfrac{5}{3}-\tfrac{19 t}{6}) u^{11}+(\tfrac{49}{6}-\tfrac{43 t}{3}) u' u^{10}+(\tfrac{287}{18}-28 t) u'^2 u^9+(\tfrac{361}{18}-35 t) u'^3 u^8 \nonumber \\ 
& \qquad +(\tfrac{208}{9}-\tfrac{241 t}{6}) u'^4 u^7+(\tfrac{475}{18}-46 t) u'^5 u^6+(\tfrac{64}{3}-\tfrac{223 t}{6}) u'^6 u^5+(\tfrac{23}{6}-7 t) u'^7 u^4 \nonumber \\ 
& \qquad +(\tfrac{65 t}{3}-\tfrac{227}{18}) u'^8 u^3+(\tfrac{74 t}{3}-\tfrac{259}{18}) u'^9 u^2+(\tfrac{65 t}{6}-\tfrac{58}{9}) u'^{10} u+(\tfrac{5 t}{3}-\tfrac{19}{18}) u'^{11} \nonumber \, , \\
a'_{24}& =  (16-28 t) u^8+(104-180 t) u' u^7+(\tfrac{908}{3}-524 t) u'^2 u^6+(524-908 t) u'^3 u^5 \nonumber \\ 
& \qquad +(\tfrac{1780}{3}-1028 t) u'^4 u^4+(448-776 t) u'^5 u^3+(\tfrac{652}{3}-376 t) u'^6 u^2 \nonumber \\ 
& \qquad +(60-104 t) u'^7 u+(\tfrac{20}{3}-12 t) u'^8 \nonumber \, , \\
a'_{31}& = \tfrac{3 t u^4}{4}-\tfrac{3 u' u^3}{4}+\tfrac{u'^2 u^2}{4}+\tfrac{3 u'^3 u}{4}+(-\tfrac{3 t}{4}-\tfrac{1}{4}) u'^4 \nonumber \, , \\
a'_{32}& =  6 u-6 u'\nonumber \, , \\
a'_{41}& =  (-\tfrac{13 t}{32}-\tfrac{3}{16}) u^7+(\tfrac{7 t}{32} +\tfrac{1}{8}) u' u^6+(-\tfrac{13 t}{32}-\tfrac{31}{96}) u'^2 u^5+(\tfrac{3 t}{32}+\tfrac{19}{96}) u'^3 u^4 \nonumber \\ 
& \qquad +(\tfrac{t}{8}+\tfrac{1}{48}) u'^4 u^3+(\tfrac{t}{16}-\tfrac{1}{12}) u'^5 u^2+(\tfrac{t}{8}+\tfrac{11}{96}) u'^6 u+(\tfrac{3 t}{16}+\tfrac{13}{96}) u'^7 \nonumber \, , \\
a'_{42}& = (\tfrac{3 t}{4}-\tfrac{1}{4}) u^4+2 u' u^3-3 u'^2 u^2+\tfrac{5 u'^3 u}{2}+(-\tfrac{3 t}{4}-\tfrac{5}{4}) u'^4 
\, .
\end{align}
This is a graded matrix factorisation when accompanied by the $U$-matrix
\be
U_{A'}(\beta) =
\operatorname{diag} \left(1, \, \E^{-\I \beta/2}, \, \E^{\I \beta/3}, \, \E^{-\I \beta/6} \right)
. 
\ee 

In anticipation of the comparison to conformal field theory in Section \ref{sec:CFTcomparison} we single out one of the roots of \eqref{eq:stE6},
\be
	t_{\text{cft}} = -1/\sqrt{3} \, .
\ee
Choosing $t=t_\text{cft}$ and performing a series of row and column manipulations produces the isomorphism 
\be\label{eq:AMF-E6}
\begin{pmatrix}
\phi & 6 & 0 & 0 \\
-\tfrac{1}{32} (12t+7) & 0 & 0 & 0 \\
0 & 0 & 1 & 0 \\
0 & 0 & \psi & 1
\end{pmatrix}
: 
A' \lra 
P_{\{0\}} \oplus P_{\{-3,-2,\ldots,3\}} 
\ee
where $\phi=\frac{1}{4} ( u+u' ) ( 3 t u^2 -3uu'+u'^2+3t u'^2 )$ and $\psi=\frac{1}{24} ( u^3-3t u^3-7 u^2 u' -3 t ^2 u'+5 u u'^2-3 t u u'^2-5 u'^3-3 t u'^3 )$. 
This proves the first third of Corollary~\ref{cor:Etypemod}. 

\begin{remark}\label{rem:galE6}
One may wonder what 
the monoid 
$X_{u'}^\dagger \otimes X_u$ 
reduces to for the other solution $t = 1/\sqrt{3}$ of~\eqref{eq:stE6}, i.\,e.~what is the equivalent to the right-hand side of~\eqref{eq:AMF-E6}. As we will see momentarily, the other solution can be related to $t_\text{cft}$ via the action of an appropriate Galois group. 

Define $\zeta_d = \E^{2 \pi \I /d}$ and consider the cyclotomic field $k = \Q(\zeta_d)$, i.\,e.~the field obtained from~$\Q$ by adjoining the $d$-th primitive root of unity~$\zeta_d$. The Galois group is isomorphic to the group of units in~$\Z_d$, $\mathrm{Gal}(k/\Q) \cong \Z_d^\times$. Given $\nu \in \Z_d^\times$ the action of the corresponding Galois group element $\sigma_\nu$ is $\sigma_\nu(\zeta_d^a) = \zeta_d^{\nu a}$.

Let now $V \in \Q[x]$ be a potential with rational coefficients and let~$M$ be a finite-rank matrix factorisation of~$V$ over~$k[x]$. We may take $M = (k[x]^{2r}, d_M)$, where~$d_M$ is a matrix with entries in~$k[x]$. Let $\sigma \in \mathrm{Gal}(k/\Q)$ be an element of the Galois group and denote by $\sigma(d_M)$ the matrix obtained by applying~$\sigma$ to each entry. Since $\sigma(V) = V$, $\sigma(d_M)$ is still a factorisation of~$V$, and we set $\sigma(M) = (k[x]^{2r}, \sigma(d_M))$. Analogously, if $f : M \to N$ is a morphism with entries in $k[x]$, then $\sigma(f)$ is a morphism from $\sigma(M)$ to $\sigma(N)$.

Let us apply this to the isomorphism in \eqref{eq:AMF-E6}. Write $A'(t)$ for~$A'$ to highlight the $t$-dependence. We choose $k = \Q(\zeta_{12})$ so that all $P_S$ are matrix factorisations over $k[u,u']$. Since
\be
	t_\text{cft} = -\tfrac13 \big( \zeta_{12} + \zeta_{12}^{-1} \big) \, , 
\ee
also $A'(t_\text{cft})$ has entries with coefficients in~$k$. The same holds for the isomorphism~\eqref{eq:AMF-E6}, and so we get an isomorphism from $\sigma(A'(t_\text{cft}))$ to $\sigma(P_{\{0\}} \oplus P_{\{-3,-2,\ldots,3\}})$. 
Since the entries of $A'(t)$ are polynomials in $u,u',t$ with rational coefficients, we have $\sigma(A'(t_\text{cft})) = A'(\sigma(t_\text{cft})))$, and $\sigma(P_S) = P_{\sigma_*(S)}$, where $\sigma_*$ is the permutation of $\Z_{12}$ induced by the action of~$\sigma$ on the $12$-th roots of unity.

It turns out that the orbit of $t_\text{cft}$ under $\mathrm{Gal}(k/\Q)$ covers all, namely both, roots of~\eqref{eq:stE6}: we have $\sigma_5(t_\text{cft}) = -\tfrac13 ( \zeta^5_{12} + \zeta_{12}^{-5} ) = 1/\sqrt{3}$ and $\sigma_{5*}(\{-3,-2,\dots,3\}) = \{ -5,-3,-2,0,2,3,5 \}$ and so we obtain the isomorphism
\be
	X^\dagger_{u'} \otimes X_u 
	\cong
	\big( P_{\{0\}} \oplus P_{\{-5,-3,-2,0,2,3,5\}} \big) 
	\otimes_k I_{v^2}
	\quad \text{for} \quad 
	t  = \sigma_5(t_\text{cft}) \, .
\ee
Note that by construction $A'(t_\text{cft})$ and $A'(\sigma_5(t_\text{cft}))$ are Morita equivalent, i.\,e.~the category $\modu(X^\dagger_{u'} \otimes X_u)$ does not depend on the choice of solution~$t$. The situation is analogous for the $\text{E}_7$- and $\text{E}_8$-singularities, cf.~Remarks~\ref{rem:galE7} and~\ref{rem:galE8}. 
\end{remark}

\subsubsection*{$\boldsymbol{V^{(\mathrm{E}_7)} \sim V^{(\mathrm{A}_{17})}}$}

In step~(i) we pick the matrix factorisation $X_0 = (R^2 \oplus R^2 , d_{X_0})$ of $x^3 + x y^3  - v^2$ with twisted differential
\be\label{eq:E7-calc-X0}
	d_{X_0}^1 = \begin{pmatrix} v & -x \\ x^2+y^3 & -v \end{pmatrix}
	\, , \quad
	d_{X_0}^0  = (d_{X_0}^1)^\#
	\, ,
\ee
one of the seven indecomposable objects of $\hmfgr(\C[v,x,y],V^{(\mathrm{E}_7)} - v^2)$ listed in \cite{kst0511155}. 
In step (iii) we again choose a similarity transformation that removes the term $x u^3$ in the diagonal entries. The result of step (iv) reads
\begin{align} \label{eq:E7-calc-solution}
d_{X_u}^1 = \big(\begin{smallmatrix} a&b \\ c & d \end{smallmatrix}\big) \quad \text{with} \quad 
	a &= v -\tfrac{t^2- 10 t + 19}2 \, (s u)^9  + (t{-}2) \, y (s u)^5  + y^2 (s u) \, ,
	\nonumber \\
	b &= -x + (2 t{-}5) \, (s u)^6 + y (s u)^2  \, ,
	\nonumber \\
	c &= x^2+y^3 + (2 t{-}5)^2 \, (s u)^{12} + (2 t{-}5) \, x (s u)^6  
	\nonumber \\
	& \qquad\quad + 2 (2 t{-}5) \, y  (s u)^8 + x y (s u)^2 + y^2 (s u)^4  \, ,
	\nonumber \\
	d &=  -v - \tfrac{t^2- 10 t + 19}2 \, (s u)^9  + (t{-}2) \, y (s u)^5  + y^2 (s u)
\end{align}
and $d_{X_0}^0  = (d_{X_0}^1)^\#$. This time~$s$ and~$t$ can be any solution of
\be\label{eq:stE7}
	t^3 - 21 \, t + 37 = 0
	\, , \quad	
	s^{18} = 26220 \, t^2 + 67488 \, t - 376912  
	\, .
\ee

We find that 
\be
	\diml(X_u) = -2s
	\, , \quad
	\dimr(X_u) = (-30 + 5 \, t + 2 \, t^2) s^{-1}
	 \, ,
\ee
and the the possible $\Q$-gradings on~$X_u$ are given by
\be
U_{X_u} (\beta)= 
\E^{\I \alpha \beta} \cdot 
\operatorname{diag} \left( 1, \, \E^{-\I \beta/3}, \, 1, \, \E^{-\I \beta/3} \right)
\quad \text{with} \quad 
\alpha \in \Q
\, .
\ee

\begin{lemma}\label{lem:E7-calc-summary}
Let~$Y$ be a rank-three $\Q$-graded matrix factorisation of $x^3 + xy^3 - u^{18} - v^2$. Suppose that (i)~$Y$ has non-zero quantum dimensions and (ii) when setting~$u$ to zero, $Y$ is equal to $X_0$ in~\eqref{eq:E7-calc-X0}. Then~$Y$ is isomorphic in $\hmfgr(\C[u,v,x,y],x^3 + xy^3 - u^{18} - v^2)$ to $X_u$ in~\eqref{eq:E7-calc-solution} for some solution $s,t$ of~\eqref{eq:stE7}.
\end{lemma}

\medskip

As before we compute the endomorphisms of~$X_u$ in $\hmfgr(\C[u,v,x,y],x^3 + xy^3 - u^{18} - v^2)$ to be the 27-dimensional space 
\be\label{eq:EndhmfX-E7}
\End
(X_u) 
\cong 
\Big(
\bigoplus_{i=0}^{16} \C_{i\cdot \frac{2}{18}}
\Big)
\oplus 
\Big(
\bigoplus_{i=4}^{12} \C_{i\cdot \frac{2}{18}}
\Big)
\oplus
\C_{\frac{8}{18}}
\, , 
\ee
and using the code of \cite{khovhompaper} together with row and column manipulations as in the previous example, for the solution
\be 
t_\text{cft} = 3 \big( \zeta_{18} + \zeta^{-1}_{18}\big) - 2 \big(\zeta^2_{18} + \zeta^{-2}_{18}\big) 
\, \quad  \text{where} \quad \zeta_d = \E^{2 \pi \I / d} \, ,
\ee 
to~\eqref{eq:stE7} we find that $X_{u'}^\dagger \otimes X_u$ is isomorphic, up to the factor $I_{v^2}$, to the rank-three graded matrix factorisation
\be\label{eq:AMF-E7}
P_{\{0\}} \oplus P_{\{-4,-3,\ldots,4\}} \oplus P_{\{-8,-7,\ldots,8\}} \, ,
\ee
proving the second third of Corollary~\ref{cor:Etypemod}. 

\begin{remark}\label{rem:galE7}
As in Remark \ref{rem:galE6} one can check that the other solutions~$t$ to~\eqref{eq:stE7} form a single orbit under the Galois group $\mathrm{Gal}(\Q(\zeta_{18})/\Q)$: the other two solutions are
\begin{align}
% t' & = 7 \big(\I  \tfrac{\sqrt{3}}{2} -\tfrac{37}{2}\big)^{-1/3}+\big(\I \tfrac{\sqrt{3}}{2} - \tfrac{37}{2} \big)^{1/3} \nonumber \, , \\
% t'' & = -\tfrac{1}{2} \big(1-\I  \sqrt{3}\big) \big(\I  \tfrac{\sqrt{3}}{2} -\tfrac{37}{2}\big)^{1/3} - \tfrac{7}{2} \big(1+\I  \sqrt{3}\big)
% \big(\I \tfrac{\sqrt{3}}{2} - \tfrac{37}{2} \big)^{-1/3}
\sigma_5(t_\text{cft}) &= 3 \big( \zeta_{18}^5 + \zeta^{-5}_{18}\big) - 2 \big(\zeta^{10}_{18} + \zeta^{-10}_{18}\big) \, , \nonumber 
\\
\sigma_7(t_\text{cft}) &= 3 \big( \zeta_{18}^7 + \zeta^{-7}_{18}\big) - 2 \big(\zeta^{14}_{18} + \zeta^{-14}_{18}\big) \, .
\end{align}
Computing the actions of $\sigma_{5*}$ and $\sigma_{7*}$ on $\{-4,-3,\ldots,4\}$ and $\{-8,-7,\ldots,8\}$, e.\,g.~$\sigma_{5*}(\{-4,-3,\ldots,4\}) = \{0,\pm 2,\pm 3,\pm 5,\pm 8\}$, one thus finds that 
\begin{align}
X^\dagger_{u'} \otimes X_u & \cong 
\big( P_{\{0\}} \oplus P_{\lbrace 0,\pm 2,\pm 3,\pm 5,\pm 8 \rbrace}\oplus P_{\{-8,-7,\ldots,8\}} \big) \otimes_k I_{v^2}
&  \text{for }   t = \sigma_5(t_\text{cft}) \, , \nonumber
\\
X^\dagger_{u'} \otimes X_u & \cong 
\big( P_{\{0\}} \oplus P_{\lbrace 0,\pm 3,\pm 4,\pm 7,\pm 8 \rbrace}\oplus P_{\{-8,-7,\ldots,8\}} \big) \otimes_k I_{v^2}
&  \text{for }  t = \sigma_7(t_\text{cft}) \, .
\end{align}
\end{remark}

\subsubsection*{$\boldsymbol{V^{(\mathrm{E}_8)} \sim V^{(\mathrm{A}_{29})}}$}

The $\mathrm{E}_8$-case is considerably more complicated than the cases already treated. This starts already in step (i) as the smallest factorisations of $V^{(\mathrm{E}_8)} - v^2$ are of rank four. Let us choose
\be\label{eq:E8-calc-X0}
	d_{X_0}^1 = \begin{pmatrix}
		 -v & 0 & x & y \\
		0 & -v & y^4 & -x^2 \\
		x^2 & y & -v & 0 \\
		y^4 & -x & 0 & -v 
		\end{pmatrix}
		,\quad
	d_{X_0}^0 = \begin{pmatrix}
		v & 0 & x & y \\
		0 & v & y^4 & -x^2 \\
		x^2 & y & v & 0 \\
		y^4 & -x & 0 & v
		\end{pmatrix}
		 ,
\ee
see \cite[Sect.\,5]{kst0511155}. This is a matrix factorisation of $x^3 + y^5 - v^2$.

In step (ii), the most generic homogeneous deformation of $d_{X_0}^1$ (deforming also the zero entries, of course) has 82 free parameters. Via the similarity transformation in step (iii) one can reduce this to 60 parameters. We refrain from giving this general deformation explicitly. 

For step (iv) one has to use a different method than in the other cases, because now $\det(d_{X_0}^1) = (x^3 + y^5 - v^2)^2$ and imposing $\det(d_{X_u}^1) = (x^3 + y^5 - u^{30}-v^2)^2$ turns out to be impractical as it results in too many non-linear conditions. Instead, we make the ansatz $d_{X_u}^0 = (x^3 + y^5 - u^{30}-v^2)^{-1} \cdot (d_{X_u}^1)^\#$ and require that $d_{X_u}^0$ be a matrix with polynomial entries. As before, this leads to a (very long) matrix factorisation of $x^3 + y^5 - u^{30}-v^2$ in terms of two parameters $s,t$, where $t$ satisfies an eighth order equation and $s^{30}$ is equal to some polynomial in $t$. The eighth order equation, however, is a product of two fourth order ones, and we select one of these and use it to simplify the matrix factorisation. One is left with the matrix
$m := d_{X_u}^1$, where with $\varsigma=su$ the matrix entries $m_{ij}$ are as follows: 
\begin{align*}
m_{11} &= - v - \tfrac{(1 + t) (3 + t) (5 + 7t)}{64}  \varsigma^{15} - \tfrac{1 + t}4 \varsigma^5 x - \tfrac{19 + 47 t + 25 t^2 + 5 t^3}{192} \varsigma^9 y - \tfrac12 \varsigma^3 y^2 \, , \\
m_{12} &= \varsigma\, , \\
m_{13} &= x + \tfrac{(-1 + t) (23 + 36 t + 5 t^2)}{96} \varsigma^{10}\, , \\
m_{14} &=  y\, , \\
m_{21} &= \tfrac{-138089 - 562209 t - 600371 t^2 - 116355 t^3}{11520}) \varsigma^{29} 
+ \tfrac{-73 - 280 t - 285 t^2 - 50 t^3}{160} \varsigma^{19} x
\\ & \qquad
+ \tfrac{-29 - 25 t + 25 t^2 + 5 t^3}{96} \varsigma^9 x^2 
+ \tfrac{-2107 - 8545 t - 9085 t^2 - 1735 t^3}{960} \varsigma^{23} y  
\\ & \qquad
+ \tfrac{-33 - 57 t - 11 t^2 + 5 t^3}{64} \varsigma^{13} x y
+ \tfrac{(5 + 7 t) (13 + 36 t + 7 t^2)}{384} \varsigma^{17} y^2  
- \tfrac{3 + 4 t}{4} \varsigma^7 x y^2 
\\ & \qquad 
+ \tfrac{-35 - 49 t + 7 t^2 + 5 t^3}{96} \varsigma^{11} y^3
- \varsigma x y^3 
- \tfrac12 (1 + t) \varsigma^5 y^4 \, , 
\\
m_{22} &= -v
+ \tfrac{(1 + t) (3 + t) (5 + 7 t)}{64} \varsigma^{15}
+ \tfrac{1 + t}{4} \varsigma^5 x
+ \tfrac{19 + 47 t + 25 t^2 + 5 t^3}{192} \varsigma^9 y
+ \tfrac12 \varsigma^3 y^2\, , \\
m_{23} &= y^4 
+ \tfrac{3587 + 14687 t + 15785 t^2 + 3125 t^3}{1920} \varsigma^{24}
+ \tfrac{(1 - t) (23 + 36 t + 5 t^2}{96} \varsigma^9 v
\\ & \qquad
+ \tfrac{43 + 102 t + 67 t^2 + 12 t^3}{96} \varsigma^{14} x
- \tfrac{(1 + t) (81 + 126 t + 17 t^2)}{384} \varsigma^{18} y
+ \tfrac{2 + 3 t}{4} \varsigma^8 x y
\\ & \qquad
+ \tfrac{(2 + t) (7 + 6 t - 5 t^2}{96} \varsigma^{12} y^2
+ \varsigma^2 x y^2 
+ \tfrac{1 + 2 t}{4} \varsigma^6 y^3\, , 
\\
m_{24} &= - x^2
+ \tfrac{(-1 + t) (23 + 36 t + 5 t^2)}{96} \varsigma^{10} x
+ \tfrac{2 + 21 t + 32 t^2 + 9 t^3}{48} \varsigma^{14} y\, , 
\\
m_{31} &= x^2
+ \tfrac{(1 - t) (23 + 36 t + 5 t^2)}{96} \varsigma^{10} x
- \tfrac{2 + 21 t + 32 t^2 + 9 t^3}{48} \varsigma^{14} y \, , 
\\
m_{32} &= y\, , \\
m_{33} &= - v
+ \tfrac{-37 - 39 t + 29 t^2 + 15 t^3}{192} \varsigma^{15} 
+ \tfrac{1 + t}4 \varsigma^5 x
+ \tfrac{-65 - 73 t + 37 t^2 + 5 t^3}{192} \varsigma^9 y
- \tfrac12 \varsigma^3 y^2\, , 
\\
m_{34} &= 
\tfrac{(1 - t) (23 + 36 t + 5 t^2)}{96} \varsigma^{11}
+ \varsigma x 
+ \tfrac{1 + t}2 \varsigma^5 y \, , 
\\
m_{41} &= y^4
+ \tfrac{3587 + 14687 t + 15785 t^2 + 3125 t^3}{1920} \varsigma^{24}
+ \tfrac{(-1 + t) (23 + 36 t + 5 t^2)}{96} \varsigma^9 v
\\ & \qquad
+ \tfrac{43 + 102 t + 67 t^2 + 12 t^3}{96} \varsigma^{14} x
- \tfrac{(1 + t) (81 + 126 t + 17 t^2}{384} \varsigma^{18} y
+ \tfrac{2 + 3 t}4 \varsigma^8 x y
\\ & \qquad
+ \tfrac{(2 + t) (7 + 6 t - 5 t^2)}{96} \varsigma^{12} y^2
+ \varsigma^2 x y^2 
+ \tfrac{1 + 2 t}4 \varsigma^6 y^3\, , \\
m_{42} &= -x 
+ \tfrac{(1 - t) (23 + 36 t + 5 t^2)}{96} \varsigma^{10} \, , 
\\
m_{43} &= 
-\tfrac{569 + 2615 t + 2855 t^2 + 425 t^3}{1920} \varsigma^{19}
+ \tfrac{17 + t - 37 t^2 - 5 t^3}{96} \varsigma^9 x
+ \tfrac{-17 - 17 t + 13 t^2 + 5 t^3}{64} \varsigma^{13} y 
\\ & \qquad
- \tfrac{1 + 2 t}4 \varsigma^7 y^2
- \varsigma y^3 \, , 
\\
m_{44} &= -v
+ \tfrac{37 + 39 t - 29 t^2 - 15 t^3}{192} \varsigma^{15}
- \tfrac{1 + t}4 \varsigma^5 x  
+ \tfrac{65 + 73 t - 37 t^2 - 5 t^3}{192} \varsigma^9 y
+ \tfrac12  \varsigma^3 y^2 \, .
\end{align*}
The parameters $s,t$ can be any solution of
the equations
\begin{align}
& s^{30} = \tfrac{1}{4} (45308593275 \, t^3 - 32199587625 \, t^2 - 973905678975 \, t - 395277903075) 
\, , \nonumber 
\\
& 5 \, t^4 - 110 \, t^2 - 120 \, t - 31 = 0 \, .
\label{eq:stE8}
\end{align}

As already noted above, the matrix $d_{X_u}^0$ is given by $(x^3 + y^5 - u^{30}-v^2)^{-1} \cdot (d_{X_u}^1)^\#$, which has polynomial entries provided $s,t$ solve \eqref{eq:stE8}. 
It is now straightforward to determine the $\Q$-gradings on~$X_u$ to be given by
\be
U_{X_u} (\beta)= 
\E^{\I \alpha \beta} \cdot 
\operatorname{diag} \left( 
1, \, 
\E^{-14\I \beta/15}, \, 
\E^{-\I \beta/3} , \, 
\E^{-3\I \beta/5} , \, 
1, \, 
\E^{-14\I \beta/15}, \, 
\E^{-\I \beta/3} , \, 
\E^{-3\I \beta/5} 
\right)
\ee
with $\alpha \in \Q$, 
and compute the quantum dimensions to be
\be
	\diml(X_u) = 2s
	\, , \quad
	\dimr(X_u) = \tfrac{5}{16} (-27 - 86 \, t - 3 \, t^2 + 4 \, t^3)  \, s^{-1} \ ,
\ee
which are indeed non-zero for all choices of $s,t$. 
This concludes the proof of Theorem~\ref{thm:ADEorbifolds}. 

Because we have made a number of restricting assumptions before arriving at $(X,d_{X_u})$, we cannot claim a statement analogous to Lemmas \ref{lem:D-calc-summary}, \ref{lem:E6-calc-summary} and \ref{lem:E7-calc-summary}. 

\medskip

Finally we compute the endomorphisms of~$X_u$ in $\hmfgr(\C[u,v,x,y],x^5 + y^3 - u^{30} - v^2)$ to be the 60-dimensional space 
\be\label{eq:EndhmfX-E8}
\End
(X_u) 
\cong 
\Big(
\bigoplus_{i=0}^{28} \C_{i\cdot \frac{2}{30}}
\Big)
\oplus 
\Big(
\bigoplus_{i=5}^{23} \C_{i\cdot \frac{2}{30}}
\Big)
\oplus 
\Big(
\bigoplus_{i=9}^{19} \C_{i\cdot \frac{2}{30}}
\Big)
\oplus
\C_{\frac{28}{30}}
\, .
\ee
For the solution
\be
 %	t_\text{cft} = \sqrt{5} - (6 + 6/\sqrt{5})^{1/2}
	t_\text{cft} = -\tfrac15 \Big( 7 + 4 \big(\zeta_{30} + \zeta_{30}^{-1} \big) + 8 \big(\zeta_{30}^2 + \zeta_{30}^{-2} \big) - 16 \big(\zeta_{30}^3 + \zeta_{30}^{-3} \big) \Big)
\ee
of~\eqref{eq:stE8}, $X_{u'}^\dagger \otimes X_u$ is isomorphic, up to the factor $I_{v^2}$, to the rank-four graded matrix factorisation
\be\label{eq:AMF-E8}
P_{\{0\}} \oplus P_{\{-5,-4,\ldots,5\}} \oplus P_{\{-9,-8,\ldots,9\}} \oplus P_{\{-14,-13,\ldots,14\}} \, .
\ee
This concludes the proof of Corollary~\ref{cor:Etypemod}. 

\begin{remark}\label{rem:galE8}
The other solutions of~\eqref{eq:stE8} are found via the Galois group $\mathrm{Gal}(\Q(\zeta_{30})/\Q)$ to be 
\begin{align}
\sigma_7(t_\text{cft}) &= 
	-\tfrac15 \Big( 7 + 4 \big(\zeta_{30}^{7} + \zeta_{30}^{-7} \big) + 8 \big(\zeta_{30}^{14} + \zeta_{30}^{-14} \big) - 16 \big(\zeta_{30}^{21} + \zeta_{30}^{-21} \big) \Big) \, , \nonumber \\
\sigma_{11}(t_\text{cft}) &= 
	-\tfrac15 \Big( 7 + 4 \big(\zeta_{30}^{11} + \zeta_{30}^{-11} \big) + 8 \big(\zeta_{30}^{22} + \zeta_{30}^{-22} \big) - 16 \big(\zeta_{30}^3 + \zeta_{30}^{-3} \big) \Big) \, , \nonumber \\
\sigma_{13}(t_\text{cft}) &= 
	-\tfrac15 \Big( 7 + 4 \big(\zeta_{30}^{13} + \zeta_{30}^{-13} \big) + 8 \big(\zeta_{30}^{26} + \zeta_{30}^{-26} \big) - 16 \big(\zeta_{30}^9 + \zeta_{30}^{-9} \big) \Big) \, . 
\end{align}
The corresponding decompositions of $X_{u'}^\dagger \otimes X_u$ are, with $Z = P_{\{0\}} \oplus P_{\{-14,-13,\ldots,14\}}$ and up to the factor $I_{v^2}$,
\begin{align}
& Z \oplus 
P_{\{0,\pm 2,\pm 5,\pm 7,\pm 9,\pm 14\}} \oplus P_{\{0,\pm 2,\pm 3,\pm 4,\pm 5,\pm 7,\pm 9,\pm 11,\pm 12,\pm 14\}} 
&&  \text{for }  t = \sigma_7(t_\text{cft}) \, , \nonumber \\
& Z \oplus 
P_{\{0,\pm 3,\pm 5,\pm 8,\pm 11,\pm 14\}} \oplus P_{\{0,\pm 2,\pm 3,\pm 5,\pm 6,\pm 8,\pm 9,\pm 11,\pm 13,\pm 14\}} 
&&  \text{for }  t = \sigma_{11}(t_\text{cft})\, , \nonumber \\
& Z \oplus 
P_{\{0,\pm 3,\pm 4,\pm 5,\pm 8,\pm 9,\pm 13\}} \oplus P_{\{0,\pm 1,\pm 4,\pm 5,\pm 8,\pm 9,\pm 12,\pm 13,\pm 14 \}} 
&&  \text{for }  t = \sigma_{13}(t_\text{cft}) \, .
\end{align}
\end{remark}

\subsubsection*{Another method to construct orbifold equivalences}

There is another way, slightly different from the one described at the beginning of Section~\ref{subsec:simple-sing}, of obtaining the orbifold equivalences $V^{(\mathrm{D}_{d/2+1})} \sim V^{(\mathrm{A}_{d-1})}$, $V^{(\mathrm{E}_6)} \sim V^{(\mathrm{A}_{11})}$, and $V^{(\mathrm{E}_7)} \sim V^{(\mathrm{A}_{17})}$. 
Roughly, this method starts with a matrix factorisation~$M$ of a potential~$W$ and computes the deformations \cite{Laudal1983, Siqveland2001}~$M_u$ of~$M$. Then instead of setting the deformation parameters to be solutions of the obstruction equations (which would give factorisations of~$W$), one tries to re-interpret some of the parameters~$u$ as variables of another potential~$V$ while choosing the remaining parameters such that~$M_u$ becomes a factorisation of $W-V$. 

As this method may prove useful to construct further orbifold equivalences, we illustrate it in more detail in the example of $V^{(\mathrm{E}_6)} \sim V^{(\mathrm{A}_{11})}$: 
\begin{enumerate}
\item Start with the matrix factorisation~\eqref{eq:E6-calc-X0} of $W=V^{(\mathrm{E}_6)}-v^2$ and compute its deformations, using e.\,g.~the implementation of \cite{cdr1112.3352}.\footnote{%
Of course the indecomposable object~\eqref{eq:E6-calc-X0} has no nontrivial deformations, but that does not matter here.} 
\item One obtains a 4-by-4 matrix~$D$ with polynomial entries in the variables $v,x,y$ and two deformation parameters $u_1, u_2$. By construction~$D$ squares to $W \cdot 1 + R$, where the matrix~$R$ vanishes if $u_1, u_2$ satisfy the obstruction equations. 
\item Interpret the parameter~$u_1$ as the (rescaled) variable~$u$ and choose $u_2 \in \C[u]$ such that~$D$ becomes a factorisation of $V^{(\mathrm{E}_6)}-u^{12}-v^2$. This turns out to be isomorphic to~$X_u$ of~\eqref{eq:E6-calc-solution} with $t=-1/\sqrt{3}$. 
\end{enumerate}

In the cases $W+v^2 \in \{ V^{(\mathrm{D}_{d/2+1})}, V^{(\mathrm{E}_6)}, V^{(\mathrm{E}_7)} \}$ this method is especially straightforward as the matrix~$D$ already squares to $(W+f)\cdot 1$, where~$f$ is a polynomial with leading term $u_1^{\,d/2}, u_1^{12}, u_1^{18}$, respectively. 
On the other hand, for $W = V^{(\mathrm{E}_8)}-v^2$ the square of~$D$ is a more generic matrix, rendering the problem $D^2 = (W - u^{30}) \cdot 1$ much more computationally involved.

\section{Comparison to conformal field theory}\label{sec:CFTcomparison}

In this section we describe how the orbifold equivalences established in Theorem~\ref{thm:ADEorbifolds} compare to -- and are in fact predicted by -- the correspondence between $\mathcal N=2$ supersymmetric Landau-Ginzburg models (LG models) and $\mathcal N=2$ superconformal field theories (CFTs) in two dimensions. We do not aim for a self-contained discussion, but we provide pointers to the literature where more details can be found.

\medskip

One can argue that for a given LG model there is an associated CFT, namely the infrared fixed point theory \cite{m1989,vw1989,howewest}. The Virasoro central charge of the infrared CFT is computed from the potential of the LG model via~\eqref{eq:def-Vir-central-charge}.
Renormalisation group flow also provides a map from certain boundary conditions and defects lines in the LG model to certain conformal boundary conditions and defects lines in the CFT, see e.\,g.~\cite{kl0210, bhls0305, br0707.0922}. In general, charges and correlators of fields in the LG model vary along the flow. However, by $\mathcal N=2$ supersymmetry the charges and correlators in a subsector of the LG model consisting of chiral primary fields are preserved and can be directly compared to their CFT equivalents.
Below, we will carry out this comparison for the defects~$X_u$ of Section~\ref{subsec:simple-sing} and the charge spectrum of the chiral primary fields supported on these defects.

\medskip

Conformal field theories form a bicategory, where objects are CFTs, 1-morphisms are topological defect lines and 2-morphisms are topological junction fields, see \cite{dkr1107.0495} for details. This bicategory has adjoints and is pivotal, hence the general concept of orbifold equivalence is applicable. There cannot be a topological defect joining two CFTs of different Virasoro central charge. As a consequence, equality of central charges provides a \textsl{necessary} condition for an orbifold equivalence to exist, cf.~Proposition~\ref{prop:necessary}. Contrary to the matrix factorisation framework, for CFTs a useful \textsl{sufficient} condition is known for the existence of an orbifold equivalence \cite{ffrs0909.5013}: 
\begin{quote}
Let $V$ be a rational vertex operator algebra. 
Suppose two CFTs have unique bulk vacua and their algebra of bulk fields contains $V \otimes_\C \overline V$ as a subalgebra. Then these two CFTs are orbifold equivalent.
\end{quote}
Here the bar over $V$ indicates that the second factor is embedded in the anti-holomorphic fields. 

Those LG models whose potentials define simple singularities are believed to renormalise to CFTs that contain the $\mathcal N=2$ minimal super Virasoro vertex operator algebra with the same central charge. The latter are rational, so by the above criterion all these CFTs of the same central charge are orbifold equivalent.

One may expect topological defects between two infrared CFTs to have analogues already in the corresponding LG model. Furthermore, one may expect that composition of topological defects commutes with renormalisation group flow. If so, the orbifold equivalences of infrared CFTs should exist also for LG models. This is the reason why Theorem~\ref{thm:ADEorbifolds} is expected from the above CFT criterion.

\medskip

After these qualitative considerations, we now turn to quantitative, more technical comparisons. The correspondence between topological defect lines of LG models of A-type singularities and those of the associated diagonal $\mathcal N=2$ CFTs motivates the following conjecture, whose ingredients we explain directly after stating it.

\begin{conjecture}\label{conj:tensor-equiv}
For any integer $d \geqslant 3$, consider the monoidal subcategory of $\hmfgr(\C[x,x'],x^d - x'^d)$ whose morphisms only have $\Q$-degree zero, and which is generated by $\{ P_{\{a,a+1,\dots,a+b\}} \,|\, a,b \in \Z_d\}$ with respect to tensor products and direct sums. This subcategory is monoidally equivalent to $(\mathcal{C}_{d-2}^{\mathfrak{su}(2)} \boxtimes \mathcal{C}_{2d}^{\mathfrak{u}(1)})_{\textrm{NS}}$.
\end{conjecture}

The matrix factorisations $P_S$ are those introduced in~\eqref{eq:P_S-def}. The relevant $\Q$-grading on $P_S$ follows from the R-charge discussion in \cite[Sect.\,3.2]{cr0909.4381}. By $\mathcal{C}_{k}^{\mathfrak{su}(2)}$ we denote the modular tensor category of integrable highest weight representations of the affine Lie algebra $\widehat{\mathfrak{su}}(2)$ at level~$k$. It has simple objects labelled by $\{0,1,\dots,k\}$. $\mathcal{C}_{2d}^{\mathfrak{u}(1)}$ is the modular tensor category of representations of the rational lattice vertex operator algebra built out of single free boson which has simple objects labelled by $\{0,1,\dots,2d-1\}$, and~$\boxtimes$ denotes the Deligne product, see e.\,g.~\cite{BaKiBook}. Details and references to the original literature can be found in \cite[App.\,A.2]{cr0909.4381}. 

We set $\mathcal{D} = \mathcal{C}_{d-2}^{\mathfrak{su}(2)} \boxtimes \mathcal{C}_{2d}^{\mathfrak{u}(1)}$ and denote the simple objects of~$\mathcal{D}$ by $U_{l,m}$, where $l \in \{0,1,\dots,d-2\}$ and $m \in \{0,1,\dots,2d-1\}$. The notation $(-)_{\textrm{NS}}$ refers to passing to the full subcategory generated by the $U_{l,m}$ with $l+m$ even. Via the coset construction, this describes the NS sector of the representations of the $\mathcal N=2$ minimal super Virasoro vertex operator algebra at central charge $c = 3 (1-2/d)$, see again \cite[App.\,A.2]{cr0909.4381} for details and references.\footnote{%
In \cite{cr0909.4381} the $(-)_{\textrm{NS}}$ is missing in (A.38) and (A.45). This is a typo or an error, depending on one's disposition towards the authors.
} 

\medskip

The status of Conjecture~\ref{conj:tensor-equiv} is currently as follows.
\begin{itemize}
\item[-] The functor which conjecturally provides the tensor equivalence acts on simple objects as $P_{\{a,a+1,\dots,a+b\}} \mapsto U_{b,b+2a}$. It is verified in \cite{br0707.0922} that this is compatible with the tensor product on the level of isomorphism classes.
\item[-] Some, but far from all, associativity isomorphisms were proved to be compatible in \cite{cr0909.4381}.
\item[-] For odd $d$, the conjecture will be proved in \cite{DRCR}.
\end{itemize}

\subsubsection*{Comparison of algebra objects}

According to \cite{tft1, Fjelstad:2006aw}, the full conformal field theories (with unique bulk vacuum) that can be constructed starting from a rational vertex operator algebra $V$ are parametrised by Morita classes of $\Delta$-separable symmetric Frobenius algebras in the modular tensor category of representations of $V$. The algebras relevant for the CFT describing the infrared fixed point of an LG model with ADE-type potential are predicted from \cite{g9608063, g0812.1318} to be non-trivial only in the $\mathfrak{su}(2)$ factor of $\mathcal{D}$. More specifically, they are representatives of the ADE classification of Morita classes of such algebras in $\mathcal{C}_{d-2}^{\mathfrak{su}(2)}$ given in \cite{o0111139}. As objects in~$\mathcal{D}$ these algebras are
\begin{align} \label{eq:CFT-ADE-list}
F^{(\textrm{A}_{d-1})} &= U_{0,0} && \text{for $d \geqslant 2$} \, ,
\nonumber \\
F^{(\textrm{D}_{d/2+1})} &= U_{0,0} \oplus U_{d-2,0}  && \text{for $d\in 2\Z_+$} \, ,
\nonumber \\
F^{(\textrm{E}_{6})} &= U_{0,0} \oplus U_{6,0}  && \text{for $d=12$} \, ,
\nonumber \\
F^{(\textrm{E}_{7})} &= U_{0,0} \oplus U_{8,0} \oplus U_{16,0} && \text{for $d=18$} \, ,
\nonumber \\
F^{(\textrm{E}_{8})} &= U_{0,0} \oplus U_{10,0} \oplus U_{18,0} \oplus U_{28,0} && \text{for $d=30$} \, .
\end{align}

Topological defects between a diagonal, i.\,e.~A-type, CFT and another CFT from the above list with the same value of $d$ given by some algebra $F$ are described by $F$-modules in $\mathcal{D}$ \cite{Frohlich:2006ch}. This is consistent with the point of view of orbifold equivalences. There, the algebra describing the theory on one side of a topological defect $X$ in terms of the other is $X^\dagger \otimes X$, see \cite{cr1210.6363} for details. In the present setting one has $X = {}_FF$, $X^\dagger = F_F$, and $F_F \otimes_F {}_FF \cong F$ as algebras. 

On the matrix factorisation side, the objects underlying the algebras describing the D- and E-type singularities as orbifolds of A-type singularities are those in \eqref{eq:PPinDcase}, \eqref{eq:AMF-E6}, \eqref{eq:AMF-E7} and~\eqref{eq:AMF-E8} above. Under the tensor equivalence of Conjecture \ref{conj:tensor-equiv}, they are indeed mapped to the corresponding objects in the list~\eqref{eq:CFT-ADE-list}. (It would of course be enough to land in the same Morita class, but for our choices of matrix factorisations we get the actual representatives chosen in \eqref{eq:CFT-ADE-list}.)

\subsubsection*{Comparison of defect spectra}

Given the above observations on the objects underlying the algebras establishing the orbifold equivalences, it is natural to expect that the matrix factorisations of the potential differences described in Section \ref{subsec:simple-sing} get mapped to the topological defects described by the modules ${}_FF$ for $F$ the corresponding algebra in \eqref{eq:CFT-ADE-list}. This expectation can be tested by comparing the spectra of chiral primaries. 

The chiral primaries in $\mathcal{D}$ are the ground states in the representations labelled $U_{l,l}$ with $l \in \{0,1,\dots,d-2\}$; their charge is $l/d$. The space of chiral primaries of holomorphic and antiholomorphic labels $l$ and $m$, respectively, on the defect ${}_FF$ is isomorphic to the vector space
\be
\Hom_F(F \otimes U_{l,l} \otimes U_{m,-m}, F) \cong \Hom(U_{l,l} \otimes U_{m,-m}, F) \, , 
\ee
where `$\Hom_F$' comprises only $F$-module maps in $\mathcal{D}$, and `$\Hom$' all maps in $\mathcal {D}$. We refer to \cite{Frohlich:2006ch} for an explanation of this formula. 
The total charge (as seen on the LG side) of these chiral primaries is $(l+m)/d$.

\medskip

Let us consider the example $F = F^{(\textrm{E}_{6})}$ in some detail. The first thing to note is that $\Hom(U_{l,l} \otimes U_{m,-m}, U_{n,0})$ has dimension zero unless $l=m$. (Actually the fusion rules in $\mathcal{C}_{2d}^{\mathfrak{u}(1)}$ tell us that $l-m \cong 0 \mod 2d$, but for the given range on~$l$ and~$m$ this just amounts to $l=m$.) For $l=m$, the space is one-dimensional if~$n$ is even and $n/2 \leqslant l \leqslant d-2-n/2$, as follows from the fusion rules of $\mathcal{C}_{d-2}^{\mathfrak{su}(2)}$.
From the summand $U_{0,0}$ in $F^{(\textrm{E}_{6})}$ we therefore get one state at each charge $l/6$, $l \in \{0,1,\dots,10\}$, and from the summand $U_{6,0}$ another state of charge $l/6$ for each $l \in \{3,4,\dots,7\}$. 
Hence the total dimension of the space of chiral primaries is 16, and we have perfect agreement with~\eqref{eq:EndhmfX-E6}. 

It is straightforward to carry out the analogous computations for $F = F^{(\textrm{D}_{d/2+1})}$, $F = F^{(\textrm{E}_{7})}$ and $F = F^{(\textrm{E}_{8})}$, again consistent with the charges and multiplicities of the matrix factorisation results listed in~\eqref{eq:PPinDcase}, \eqref{eq:EndhmfX-E7} and~\eqref{eq:EndhmfX-E8}, respectively.

\end{document}